\documentclass[11pt]{article}
\usepackage{amssymb}
\usepackage{epsfig}
\usepackage{amsmath}
\usepackage{natbib}
\usepackage{latexsym}
\usepackage{amsfonts}
\usepackage{color}
\usepackage{graphicx}

\usepackage{amsthm}

\numberwithin{equation}{section}

\parskip \baselineskip
\newtheorem{lemma}{Lemma}
\newtheorem{proposition}{Proposition}
\newtheorem{theorem}{Theorem}
\newtheorem{corollary}{Corollary}
\newtheorem{remark}{Remark}
\newtheorem{definition}{Definition}
\newtheorem{example}{Example}

\begin{document}

\begin{center}
{\Large \bf Stopping models closed under pgf composition, and the stability of randomly stopped model extensions}
\end{center}

\begin{center}
{\large Jordi Valero, Josep Ginebra\footnote{
       Address for correspondence: Department of Statistics, Universitat Polit\`ecnica de Catalunya, Avgda.
       Diagonal 647, 6$^{\hbox{\rm a}}$ Planta, 08028 Barcelona, Spain
       (E-mail: josep.ginebra@upc.edu)}
        }
\end{center}

Statistical model transformations based on randomly stopped sums, maxima, and minima are widely used to
extend statistical models by introducing an additional parameter.
We characterize the finite-dimensional positive stopping models for which these randomly stopped sum and extreme
transformations act as statistically stable (idempotent) model extensions.
Stability requires the underlying stopping model to be closed under pgf composition. We prove that
any finite-dimensional, connected, positive stopping model closed under pgf composition necessarily forms a family of random variables whose pgfs commute.
Using the associated Koenigs function, we establish that these models
form a statistical manifold admitting a global, one-dimensional parametrization $\theta = \Pr(N=1) \in (0, \theta_*]$, in which
the probability mass at $i$ is a polynomial in $\theta$ of degree at most $i$.
Within this class, we further establish a duality between those models containing
the identity variable (the ones yielding statistically stable extensions) and the set of all probability distributions
supported on the positive integers. By providing numerous examples of stopping models closed under pgf composition,
we clarify the assertion that, under certain conditions, the only stopping model that is so is geometric.



\section{Introduction}		

The transformation of parametric statistical models through randomly stopped sums, introduced by  Neyman (1939) and Feller (1943),
is a cornerstone of mathematical statistics and stochastic modeling. The compound models obtained this way
include some of the most versatile models in applied statistics, and are used
across disciplines like engineering, natural science, actuarial science, mathematical finance, human-computer interaction and queueing theory.
Similarly, transformations based on randomly stopped maxima or minima,
introduced by Shaked (1975) and Marshall and Olkin (1997), have become essential for constructing flexible lifetime and heavy-tailed models
used in reliability engineering, extreme value analysis, insurance, and epidemiology.

The existing literature extensively details the distributional properties of specific models resulting
from a single application of these transformations using a given stopping model and a given stopped model, as the more than two thousand
citations of Marshall and Olkin (1997) and seven hundred citations of Neyman (1939) indicate.

Conversely, the structural consequences of the repeated application of these transformations remain less understood.
Usually, repeated use of these transformations starting from an initial model keeps ``enlarging" the successive models by increasing
the dimension of the corresponding parameter spaces.
On the other hand, the exact conditions under which the iterative use of a randomly stopped sum or extreme model transformation reaches
a stable steady state (preserving the parameter space),
have not yet been established.

The central challenge tackled here is the characterization of the conditions under which these statistical model transformations, $\mathcal{T}(\cdot)$,
work as statistically stable extensions.
Formally, given an initial parametric statistical model, $\mathcal{X}$, we seek the requirements ensuring that $\mathcal{T}$ acts as an extension, in the sense that
$\mathcal{X} \subset \mathcal{T}(\mathcal{X})$ for all $\mathcal{X}$,
and the conditions under which this extension is statistically stable,
in the sense that $\mathcal{T}$ is idempotent, i.e.,
$\mathcal{T}(\mathcal{T}(\mathcal{X})) = \mathcal{T}(\mathcal{X})$ for all $\mathcal{X}$.

We establish that the statistical stability (idempotency) of these model transformations is not a property of initial models $\mathcal{X}$,
but it instead requires that the underlying stopping model $\mathcal{N}$ be closed under pgf composition.
For the model transformation to function as an extension, the stopping model
$\mathcal{N}$ must include the degenerate random variable at one $N_I$, which we call the identity variable.

Our primary contribution, Theorem \ref{thm:ult2}, establishes that any finite-dimensional, positive stopping model that is closed under
pgf composition, connected, and has a non-empty interior is necessarily a family of random variables that commute.
Consequently, these stopping models
inherit the properties of these commuting families, detailed in Proposition \ref{prop:prpts}.
In particular, by leveraging the Schr\"oder functional equation and the corresponding Koenigs function,
we are able to treat the resulting stopping models as differentiable statistical manifolds such that:
\begin{enumerate}
\item
They admit a global, identifiable, one-dimensional parametrization through $\theta=\Pr(N=1) \in (0,\theta_*]$ for some $0 < \theta_* \le 1$;
\item
For any $i>0$, the probability mass at $i$ is a polynomial in $\theta$ of degree at most $i$, which provides a simple framework for
likelihood-based inference.
\end{enumerate}
Refining the main result, Theorem \ref{thm:nnn} establishes a duality between the subclass of these models containing $N_I$
and the set of probability distributions supported on the positive integers.
This duality completely solves the characterization problem by
identifying the unique positive stopping models that ensure that randomly stopped sum and randomly stopped extreme
model transformations function as statistically stable (idempotent) model extensions.

The broader class of composition-closed and connected stopping models that do not have to include $N_I$
allows for statistically stable extensions via a generalized framework of randomly stopped extreme transformations
combining randomly stopped maxima or minima with their inverses, introduced by Valero and Ginebra (2025).

Note that in this paper, we restrict consideration to stopping models strictly supported on the positive integers, which is
a necessary condition for randomly stopped extreme transformations.
Although randomly stopped sum transformations can accommodate stopping models with zero in their support,
such zero-inclusive models fall outside the scope of this paper.

The remainder of this paper is organized as follows. Section 2 establishes basic definitions and notation.
Section 3 examines families of random variables that commute, focusing on those with a non-empty interior.
Section 4 proves that finite-dimensional, connected, composition-closed, positive stopping models
with a non-empty interior must be families of random variables that commute,
detailing their core properties.
Section 5 characterizes the subclass of these models containing $N_I$.
Section 6 provides numerous examples of composition-closed, connected stopping models (both with and without $N_I$),
challenging the assertion in Marshall and Olkin (1997) that, under unspecified ``regularity conditions,"
the only composition-closed stopping model is geometric. Finally, Section 7 illustrates how randomly
stopped sum and extreme extensions with composition-closed stopping models yield
new families of invariant statistical models.

\section{Basic definitions and notation}

\subsection{Stopping models}

Let $N$ be a positive integer-valued random variable
with $\Pr(N=0)=0$ and probability generating function given by $h_N(t) = \sum_{i=1}^{\infty} p_i(N) t^i$,
where $p_i(N) = \Pr(N=i)$ for $i \in \mathbb{N}^{+}$. We denote by $N_I$ the degenerate random variable at one,
which satisfies $\Pr(N_I=1)=1$ and has pgf $h_{N_I}(t)=t$.

A stopping model $\mathcal{N}$ is a family of such positive integer-valued random variables.
In this paper we exclusively consider models for which $p_0(N)=\Pr(N=0)=0$ for all $N \in \mathcal{N}$.

We assume that stopping models admit a parametric structure of one of the following types:
\begin{enumerate}
\item
Global parametrization: The random variables are indexed by a parameter vector
$\delta$ in a parameter space $\mathcal{D}_\mathcal{N} \subset \mathbb{R}^{m}$, assumed to be
an open set (possibly with boundary).
In this case, we represent the model as
$\mathcal{N} = \{N_{\delta} \mid h_{N_{\delta}}(t) = \sum_{i=1}^{\infty} p_i(\delta) t^i, \quad \delta \in \mathcal{D}_\mathcal{N} \}.$
\item
Local parametrization: The model is an $m$-dimensional differentiable manifold
(possibly with boundary), embedded in the space of probability sequences.
In this case, each $N \in \mathcal{N}$ admits a local parametrization by $m$ coordinates.
\end{enumerate}
To ensure statistical regularity, we eventually focus on stopping models such that:
\begin{enumerate}
\item The probabilities $p_i(\cdot)$ are continuously differentiable with respect to the parameters for all $i \in \mathbb{N}^{+}$.
\item The model is connected and has a non-empty interior relative to its $m$-dimensional embedding space.
\end{enumerate}
Let $\mathcal{H}_\mathcal{N} = \{ h_{N} \mid N \in \mathcal{N} \}$ denote the set of pgfs associated with the stopping model.

\begin{remark}
A real-valued function $h_N(\cdot)$ is the pgf of a
positive integer-valued random variable $N$ if and only if:
(i) $h_N(0)=0$ and $h_N(1)=1$;
(ii) $h_N(\cdot)$ is analytic on $[0,1)$; and
(iii) all derivatives of $h_N(\cdot)$ are non-negative on $[0,1)$.
Consequently, every $h_{N} \in \mathcal{H}_\mathcal{N}$ is a
non-decreasing, convex bijection from $[0,1]$ onto $[0,1]$.
\end{remark}

\subsection{Stopping models closed under pgf composition}

Let $N_1 \circ N_2$ denote the
random variable whose pgf is given by the composition of the pgfs of $N_1$ and $N_2$,
namely $h_{N_1 \circ N_2}(t) = h_{N_1}(t) \circ h_{N_2}(t) = h_{N_1}(h_{N_2}(t))$.
By iteration, we define $N^{\circ m}$ as the random variable whose pgf is the $m$-fold composition of $h_{N}$ with itself.
By convention, $N^{\circ 1} = N$ and $N^{\circ 0} = N_I$, where
$N_I$ denotes the identity, with $h_{N_I}(t) = t$. Under this notation, it
follows that $N^{\circ k} \circ N^{\circ j} = N^{\circ(k+j)}$ for all $k,j \in \mathbb{N}$.

In general, the fact that a model $\mathcal{N}$ contains $N_1$ and $N_2$
does not guarantee that their composition also belongs to $\mathcal{N}$.
This observation motivates the following definition.
\begin{definition}
A stopping model $\mathcal{N}$
is closed under pgf composition if, for every pair $N_1, N_2 \in \mathcal{N}$,
their composition $N_1 \circ N_2$ also belongs to $\mathcal{N}$.
\end{definition}
This article characterizes finite-dimensional stopping models that
are closed under pgf composition, connected and with a non-empty
interior. These models are of interest because under them the statistical model
transformations presented next can work as statistically stable extensions.

\subsection{Model transformations based on randomly stopped sum or extremes}

Let $\mathcal{X} = \{X_{\gamma} \mid \gamma \in \mathcal{D}_\mathcal{X} \}$ be a statistical
model with a real-valued sample space, cumulative distribution functions (cdfs) $F_{X_{\gamma}}$,
survival functions $S_{X_{\gamma}}$,
and characteristic functions $\phi_{X_{\gamma}}$.

Consider a sequence $(X_1, \ldots, X_N)$ of independent and identically distributed copies of a random variable $X$,
where the stopping variable $N$ is independent of the sequence.
The \emph{$N$-stopped sum of $X$} is defined as $Y = \sum_{i=1}^{N} X_{i}$,
with characteristic function $\phi_{Y} = h_{N}(\phi_{X})$.
Given a stopping model $\mathcal{N} = \{N_{\delta} \mid \delta \in \mathcal{D}_\mathcal{N} \}$,
we define the following statistical model transformation.

\begin{definition}\label{def:rssm}
The $\mathcal{N}$-stopped sum of $\mathcal{X}$ is the statistical model
comprising all random variables $Y_{\gamma, \delta}=\sum_{i=1}^{N_\delta} X_{\gamma i}$ with characteristic
function $\phi_{Y_{\gamma,\delta}} = h_{N_{\delta}}(\phi_{X_{\gamma}})$,
$$ \mathcal{Y} = sum_{\mathcal{N}}(\mathcal{X}) =
\left\{Y_{\gamma, \delta} \mid \phi_{Y_{\gamma, \delta}} = h_{N_{\delta}}(\phi_{X_{\gamma}}), \quad
\gamma \in \mathcal{D}_\mathcal{X}, \delta \in \mathcal{D}_\mathcal{N} \right\}.  $$
When $\mathcal{X}$ is non-negative integer-valued with pgf $h_{X_{\gamma}}$, one also has that:
$$ \mathcal{Y} = sum_{\mathcal{N}}(\mathcal{X}) =
\left\{Y_{\gamma, \delta} \mid h_{Y_{\gamma, \delta}} = h_{N_{\delta}}(h_{X_{\gamma}}), \quad
\gamma \in \mathcal{D}_\mathcal{X}, \delta \in \mathcal{D}_\mathcal{N} \right\}.  $$
\end{definition}
The use of randomly stopped sums to transform statistical models has a very long story in the literature, dating back to
Neyman (1939), Feller (1943) and Gurland (1957, 58), and the resulting models are sometimes called compound models.
For a comprehensive review of the literature in the case where $\mathcal{X}$ is discrete, see Chapter 9 of Johnson, Kotz and Kemp (2005).

Similarly, the \emph{$N$-stopped maximum of $X$} is
$Y = \max(X_1,\cdots, X_N)$ with cdf $F_{Y} = h_{N}(F_{X})$, and
the \emph{$N$-stopped minimum of $X$} is $Y = \min(X_1, \cdots, X_N)$ with survival function $S_{Y} = h_{N}(S_X)$.
Given a stopping model $\mathcal{N}$, we define the following two statistical model transformations.

\begin{definition}\label{def:rsex}
The $\mathcal{N}$-stopped maximum of $\mathcal{X}$ is the statistical model comprising all random variables
$Y_{\gamma, \delta}= \max(X_{\gamma 1},\cdots, X_{\gamma N_\delta})$
with cdfs $F_{Y_{\gamma, \delta}} = h_{N_{\delta}}(F_{X_{\gamma}})$:
$$ \mathcal{Y} = max_{\mathcal{N}}(\mathcal{X}) =
\left\{Y_{\gamma, \delta} \mid F_{Y_{\gamma, \delta}} = h_{N_{\delta}}(F_{X_{\gamma}}), \quad
\gamma \in \mathcal{D}_\mathcal{X}, \delta \in \mathcal{D}_\mathcal{N} \right\}, $$
and the $\mathcal{N}$-stopped minimum of $\mathcal{X}$ is the model comprising all
$Y_{\gamma, \delta}= \min(X_{\gamma 1}, \cdots, X_{\gamma N_\delta})$
with survival function $S_{Y_{\gamma, \delta}} = h_{N_{\delta}}(S_{X_{\gamma}})$:
$$ \mathcal{Y} = min_{\mathcal{N}}(\mathcal{X}) =
\left\{Y_{\gamma, \delta} \mid S_{Y_{\gamma, \delta}} = h_{N_{\delta}}(S_{X_{\gamma}}), \quad
\gamma \in \mathcal{D}_\mathcal{X}, \delta \in \mathcal{D}_\mathcal{N} \right\}. $$
\end{definition}
These transformations have been considered ever since Shaked (1975) and
Shaked and Wong (1997). Marshall and Olkin (1997) considers the special case where the stopping model
is the zero-truncated geometric.

\begin{remark}
For randomly stopped extremes, stopping models $\mathcal{N}$ satisfying
$\Pr(N_\delta=0)>0$ for some $\delta \in \mathcal{D}_{\mathcal{N}}$ cannot be used, because in such cases, $h_{N_{\delta}}(0)>0$,
which causes $h_{N_{\delta}}(F_{X_{\gamma}})$ and $h_{N_{\delta}}(S_{X_{\gamma}})$ in Definition \ref{def:rsex} to be improperly defined as
cumulative distribution and survival functions.
Conversely, for randomly stopped sums, it is possible to apply Definition \ref{def:rssm} to stopping models with zero
in their support by defining the empty sum as zero. Nevertheless, throughout this article, we restrict our analysis
to stopping models strictly supported on the positive integers, which is a necessary condition for randomly stopped extreme
transformations, but leaves the implications of zero-inclusive stopping models for randomly stopped sums
outside our consideration.
\end{remark}

Finally, we also consider four statistical model transformations proposed and motivated in Valero and Ginebra (2025).
They combine randomly stopped maxima or minima with two new model transformations
using the functional inverses of the pgfs, $h_N^{-1}$, that
may be interpreted as the inverse operations of randomly stopped maxima or minima.
\begin{definition}\label{def:rscb}
Let $\mathcal{X}$ and $\mathcal{N}$ be as defined above.
We define the following
two statistical model transformations associated with randomly stopped maxima,
$$ \mathcal{Y} = max_{\mathcal{N}}^{-1} \circ max_{\mathcal{N}}(\mathcal{X}) =
\left\{Y_{\gamma, \delta_1, \delta_2} \mid F_{Y_{\gamma, \delta_1, \delta_2}} = h_{N_{\delta_2}}^{-1} \circ h_{N_{\delta_1}}(F_{X_{\gamma}}), \quad
\gamma \in \mathcal{D}_\mathcal{X},  \delta_1, \delta_2 \in \mathcal{D}_\mathcal{N} \right\}, $$
$$ \mathcal{Y} = max_{\mathcal{N}} \circ max_{\mathcal{N}}^{-1}(\mathcal{X}) =
\left\{Y_{\gamma, \delta_1, \delta_2} \mid F_{Y_{\gamma, \delta_1, \delta_2}} = h_{N_{\delta_2}} \circ h_{N_{\delta_1}}^{-1}(F_{X_{\gamma}}), \quad
\gamma \in \mathcal{D}_\mathcal{X},  \delta_1, \delta_2 \in \mathcal{D}_\mathcal{N} \right\}, $$
and the two transformations associated with randomly stopped minima,
$$ \mathcal{Y} = min_{\mathcal{N}}^{-1} \circ min_{\mathcal{N}}(\mathcal{X}) =
\left\{Y_{\gamma, \delta_1, \delta_2} \mid S_{Y_{\gamma, \delta_1, \delta_2}} = h_{N_{\delta_2}}^{-1} \circ h_{N_{\delta_1}}(S_{X_{\gamma}}), \quad
\gamma \in \mathcal{D}_\mathcal{X},  \delta_1, \delta_2 \in \mathcal{D}_\mathcal{N} \right\}, $$
$$ \mathcal{Y} = min_{\mathcal{N}} \circ min_{\mathcal{N}}^{-1}(\mathcal{X}) =
\left\{Y_{\gamma, \delta_1, \delta_2} \mid S_{Y_{\gamma, \delta_1, \delta_2}} = h_{N_{\delta_2}} \circ h_{N_{\delta_1}}^{-1}(S_{X_{\gamma}}), \quad
\gamma \in \mathcal{D}_\mathcal{X},  \delta_1, \delta_2 \in \mathcal{D}_\mathcal{N} \right\}. $$
\end{definition}

\subsection{Statistical stability of statistical model transformations}

In general, a transformed model $\mathcal{Y} = \mathcal{T}(\mathcal{X})$
need not necessarily contain the original model $\mathcal{X}$.
However, certain stopping models ensure that the transformations detailed in Subsection 2.3 act as statistical model extensions
in the sense that $\mathcal{X} \subset \mathcal{T}(\mathcal{X})$ for any initial model $\mathcal{X}$.
\begin{proposition}
The transformations based on randomly stopped sums and extremes
(Definitions \ref{def:rssm} and \ref{def:rsex}) are statistical model extensions if and only if
$N_I \in \mathcal{N}$.
 \label{prop:stex}
\end{proposition}
While recursive application of these transformations typically generates a sequence of distinct and increasingly
complex models, we focus on the characterization of the specific cases where the transformation process becomes statistically stable as defined next.
\begin{definition}
A statistical model transformation $\mathcal{T}(\cdot)$ is statistically stable if it is idempotent, i.e. if
it is such that $\mathcal{T}(\mathcal{T}(\mathcal{X}))=\mathcal{T}(\mathcal{X})$ for any statistical model $\mathcal{X}$.
 \label{def:stsb}
\end{definition}
Statistical models $\mathcal{T}(\mathcal{X})$ obtained through a statistically stable transformation $\mathcal{T}(\cdot)$ are
themselves stable models under repeated use of that transformation, in the sense that they remain invariant under successive applications of $\mathcal{T}(\cdot)$,
given that $\mathcal{T}(\mathcal{T}(\mathcal{X}))=\mathcal{T}(\mathcal{X})$.

\begin{remark}
The notion of statistical stability considered here applies to families of random variables (statistical models) and refers to
the idempotency of statistical model transformations. Our definition generalizes the
geometric-extreme stability of Marshall and Olkin (1997), which specifically refers to the idempotency of geometrically stopped
maximum and minimum transformations.
This model level property is distinct from classic probabilistic notions of stability and max-stability, which apply to individual random variables,
and mean that the linear combinations and the maximum of a number of independent copies of a random variable retain the original
distribution up to location and scale adjustments.
Our concept of statistical stability is also distinct from perturbation stability, stochastic process stability, and the stability of robust estimators.
\end{remark}

The  following result establishes that for the transformations considered here, statistical
stability is fundamentally tied to the algebraic closure of the stopping model.
The first part of the result follows directly from Definitions \ref{def:rssm} and \ref{def:rsex},
whereas the proof of the second part also relies on Proposition \ref{prop:prpts} in Section 4,
which applies only in the context of finite-dimensional connected stopping models.
\begin{proposition}
The transformations based on randomly stopped sums and extremes (Definitions \ref{def:rssm} and \ref{def:rsex})
yield statistically stable model extensions if and only if:
(i) $N_I \in \mathcal{N}$ (ensuring that the resulting model is an extension), and
(ii) $\mathcal{N}$ is closed under pgf composition (ensuring stability).

The composite transformations in Definition \ref{def:rscb} generated by an underlying finite-dimensional, connected,
positive stopping model $\mathcal{N}$ that is closed under pgf composition always yield statistically stable model extensions.
 \label{prop:stst}
\end{proposition}
Stopping models closed under pgf composition turn out to be intimately related to the families of
random variables that commute. In Section 3 we analyze these commuting families, which serve as the foundation
for the characterization in Section 4 of finite-dimensional, closed and connected stopping models.
Section 5 refines these results for the sub-class of models that include $N_I$.

\section{Stopping models formed by random variables that commute}

One approach to constructing stopping models closed under pgf composition is to identify the family of
all random variables that commute with a given reference random variable, $N_0$.
We find that if $1$ is not in the support of $N_0$, the resulting stopping model has an empty interior.
Consequently, our analysis ultimately focuses on the case where $0<p_1(N_0)<1$.
The Propositions in this section are proven in Appendix 1.

\subsection{Family of random variables that commute with $N_0$, $\mathcal{N}_{N_0}$}

We define the stopping model $\mathcal{N}_{N_0}$ as the family of all random variables that commute
with a fixed reference random variable, $N_0$.
\begin{definition} Let $N_0$ be a positive integer-valued random variable such that $N_0 \neq N_I$, with pgf $h_{N_0}(t)$.
The stopping model $\mathcal{N}_{N_0}$ is the set of
all positive integer-valued random variables $N$ whose pgfs, $h_N$, commute with $h_{N_0}$ under function composition,
$$\mathcal{N}_{N_0} = \left\{ N \mid h_N\circ h_{N_0} = h_{N_0} \circ h_N \right\}. $$
The corresponding set of pgfs is denoted by $\mathcal{H}_{N_0}$.
 \label{def:smn0}
\end{definition}
An immediate consequence of the associativity of function composition is that model $\mathcal{N}_{N_0}$ is closed and includes $N_I$.
\begin{proposition}
The stopping model $\mathcal{N}_{N_0}$ is closed under pgf composition. It contains the identity $N_I$
and the sub-model $\{N_0^{\circ m}: m \in \mathbb{N} \}$, which is itself closed under composition.
\label{prop:zn1}
\end{proposition}
The following result demonstrates that whether $1$ is in the support of $N_0$ or not determines whether
$1$ is in the support of all other elements in $\mathcal{N}_{N_0}$.
\begin{proposition}
Let $N_0$ and $N \in \mathcal{N}_{N_0}$ be non-identity positive integer-valued random variables. Then,
$p_1(N)=0$ if and only if $p_1(N_0) = 0$.
 \label{prop:znz}
\end{proposition}

We next establish that when $p_1(N_0) = 0$, the model $\mathcal{N}_{N_0}$ has an empty interior.
\begin{proposition}
If $p_1(N_0)=0$, then every $N \in \mathcal{N}_{N_0}$
can be parametrized in an identifiable manner by the order of the first non-zero term in its pgf power series.
Consequently, the parameter space of $\mathcal{N}_{N_0}$ is a subset of $\mathbb{N}^{+}.$
\label{prop:zn2}
\end{proposition}

\subsection{Set of analytic functions that commute with the pgf of $N_0$, $\mathcal{G}_{N_0}$}

We now extend the set $\mathcal{H}_{N_0}$ to a broader class by relaxing the
requirement that derivatives beyond the first must be non-negative.
This allows us to include all non-constant functions that are analytic on $[0,1)$ and commute with $h_{N_0}$,
even if they are not a pgf. That will provide the mathematical
framework for characterizing the stopping models manifold structure.
\begin{proposition}\label{prop:gincrb}
Let $N_0$ be a positive integer-valued random variable such that $N_0 \neq N_I$, with pgf $h_{N_0}$.
Any function $g$ that is analytic on $[0,1)$ and commutes with $h_{N_0}$ under composition is either:
\begin{enumerate}
\item
One of the constant functions $g(t)\equiv0$ or $g(t)\equiv1$; or
\item
A strictly increasing bijection from $(0,1]$ onto itself such that $g(0)=0$ and $g(1)=1$.
\end{enumerate}
\end{proposition}
Below we  extend $\mathcal{H}_{N_0}$ by considering all non-constant analytic functions that
commute with $h_{N_0}$.
\begin{definition}\label{def:smn1}
Let $N_0$ be a positive integer-valued random variable with $N_0 \neq N_I$.
The set $\mathcal{G}_{N_0}$ is defined to be the family of all non-constant analytic functions
on $[0,1)$ that commute with $h_{N_0}$:
$$ \mathcal{G}_{N_0} = \left\{ g \mid \mbox{is analytic and} \quad g\circ h_{N_0} = h_{N_0} \circ g \quad \mbox{on} \quad [0,1) \right\}. $$
 \end{definition}
\begin{remark}
The functional inverse $h_N^{-1}(\cdot)$ of a pgf, satisfies $h_N^{-1}(0)=0$ and $h_N^{-1}(1)=1$, and
if $h_N$ commutes with $h_{N_0}$, then so does $h_N^{-1}$. However, even though $h_N^{-1}(\cdot)$
is always analytic on $(0,1)$, its analycity at $t=0$ depends on $h_N^{\prime}(0)$. Specifically,
$h_N^{-1} \in \mathcal{G}_{N_0}$ if and only if $h_N^{\prime}(0) = p_1(N) \neq 0$. Hence, the critical role of
$p_1(N_0)$ in determining whether $\mathcal{G}_{N_0}$ is large enough to include $h_N^{-1}$.
\end{remark}
The properties established for $\mathcal{N}_{N_0}$
extend to the broader class, $\mathcal{G}_{N_0}$.
\begin{proposition}
The set $\mathcal{G}_{N_0}$ is closed under function composition.
 \label{prop:gcfc}
\end{proposition}
\begin{proposition}
Let $g \in \mathcal{G}_{N_0}$ be such that $g(t)\neq t$. Then, $g^{\prime}(0) = 0$ if and only if $p_1(N_0) = 0$.
 \label{prop:init}
\end{proposition}
\begin{proposition}
If $p_1(N_0) = 0$, the elements in $\mathcal{G}_{N_0}$ are uniquely identified
by the order of the first non-zero term in their power series expansion
around $t=0$. In that case, the parameter space of $\mathcal{G}_{N_0}$ is $\mathbb{N}^{+}$.
\label{prop:init2}
\end{proposition}

\begin{remark}
Propositions \ref{prop:zn2} and \ref{prop:init2} establish that the condition $p_1(N_0) \neq 0$ is necessary
for $\mathcal{N}_{N_0}$ and $\mathcal{G}_{N_0}$ to have a non-empty interior. Since our focus is on
stopping models connected and with a non-empty interior,
we shall assume $0<p_1(N_0)<1$ throughout the remainder of this section.
\end{remark}

\subsection{Parametrization of $\mathcal{G}_{N_0}$ and $\mathcal{N}_{N_0}$ when $0<p_1(N_0)<1$}

The primary result of this section establishes that when $1$ is in the support of $N_0$, all functions in
the commuting family $\mathcal{G}_{N_0}$ are uniquely determined by their derivative at the origin, $\theta = g^{\prime}(0)$.
This derivative serves as a global identifiable parameter.
\begin{theorem}
Let $N_0$ be a positive integer-valued random variable with $0 < p_1(N_0) < 1$.
The set of all non-constant analytic functions on $[0,1)$ that commute with $h_{N_0}$ is given by
$$ \mathcal{G}_{N_0} = \left\{g_{\theta}(t) = K_{N_0}^{-1}(\theta  K_{N_0}(t)) \mid \theta = g_{\theta}^{\prime}(0) \in (0,\infty)  \right\},     $$
where $K_{N_0}$ is the Koenigs function defined by
$$ K_{N_0}(t) = \lim_{n\rightarrow\infty}\frac{h_{N_0}^{\circ n}(t)}{[h_{N_0}^{\prime}(0)]^n}. $$
The set $\mathcal{G}_{N_0}$ admits a unique identifiable parametrization via
$\theta = g_\theta^{\prime}(0)$, and the corresponding parameter space is $(0,\infty)$.
\label{thm:par}
\end{theorem}

\begin{proof}
Since $h_{N_0}$ is analytic on $[0,1)$ with $h_{N_0}(0)=0$
and $0 < h_{N_0}^{\prime}(0) = p_1(N_0) < 1$, classical results (see, e.g., Kuczma, 1968)
state that the Schr\"oder functional equation,
$K_{N_0}\circ h_{N_0} = h_{N_0}^{\prime}(0) K_{N_0},$
admits a unique analytic solution $K_{N_0}$ under the normalization $K_{N_0}^{\prime}(0)=1$.
This Koenigs function
is analytic on $[0,1)$, satisfies $K_{N_0}(0)=0$, and $\lim_{t\rightarrow 1^{-}}K_{N_0}(t)=\infty$,
and its Taylor coefficients are non-negative.

A central result in iteration theory (Pranger, 1970, Cowen, 1984), states that a non-constant analytic function $g$
commutes with $h_{N_0}$ if and only if it satisfies the same Schr\"oder equation, $K_{N_0}\circ g = g^{\prime}(0) K_{N_0}$.
Solving for $g(t)$ yields $g(t) = K_{N_0}^{-1}(g^{\prime}(0)  K_{N_0}(t))$.

Defining $\theta=g^{\prime}(0)$ and noting that the range of $g$ remains in the domain of $K_{N_0}$ for $\theta>0$, we obtain
the family $\mathcal{G}_{N_0}$ as stated. The uniqueness of $K_{N_0}$ ensures that each $\theta$
corresponds to exactly one $g \in \mathcal{G}_{N_0}$, completing the proof.
\end{proof}

Since $\mathcal{H}_{N_0} \subset \mathcal{G}_{N_0},$
we can parametrize the statistical model $\mathcal{N}_{N_0}$ through
$\theta = p_1(N)$, which means that $\mathcal{N}_{N_0}$ admits a global unidimensional parametrization.
\begin{corollary}
Let $N_0$ be a positive integer-valued random variable with $0<p_1(N_0)<1$. Then
the statistical model $\mathcal{N}_{N_0}$
is identifiably parametrized by $\theta = h_N^{\prime}(0)  = p_1(N)$,
with parameter space $\mathcal{D}_{N_0} \subset (0,1]$.
 \label{cor:parh}
\end{corollary}
Theorem \ref{thm:par} implies that
function composition within $\mathcal{G}_{N_0}$ (and $\mathcal{H}_{N_0}$)
is isomorphic to the multiplication of their $\theta$ parameters, and that all pairs of functions
in these sets commute. Additionally, $\mathcal{G}_{N_0}$ is closed under inversion,
where the inverse of $g_\theta$ is $g_{1/\theta}$.
\begin{proposition}\label{prop:com}
Let $N_0$ be a positive integer-valued random variable with $0<p_1(N_0)<1$, and let
$g_{\theta_i} \in \mathcal{G}_{N_0}$ and $h_{N_{\theta_i}} \in \mathcal{H}_{N_0}$.
Then:
\begin{enumerate}
\item{}
Commutativity:
$$ g_{\theta_1}\circ g_{\theta_2} =
g_{\theta_2}\circ g_{\theta_1} =
g_{\theta_1 \theta_2} \in \mathcal{G}_{N_0}, \quad \mbox{and} \quad
h_{N_{\theta_1}}\circ h_{N_{\theta_2}} = h_{N_{\theta_2}}\circ h_{N_{\theta_1}} =
h_{N_{\theta_1 \theta_2}} \in
\mathcal{H}_{N_0}.   $$
\item{}
Inversion:
$$ g_{\theta_1}^{-1} = g_{1/\theta_1} \in \mathcal{G}_{N_0}, \quad \mbox{and} \quad
h_{N_{\theta_1}}^{-1} = g_{1/\theta_1} \in \mathcal{G}_{N_0}. $$
\item{}
Ratio composition:
$$ g_{\theta_1}^{-1}\circ g_{\theta_2} = g_{\theta_2}\circ g_{\theta_1}^{-1} =
g_{\theta_2/\theta_1} \in \mathcal{G}_{N_0}, \quad \mbox{and} \quad
h_{N_{\theta_1}}^{-1}\circ h_{N_{\theta_2}} = h_{N_{\theta_2}}\circ h_{N_{\theta_1}}^{-1} =
g_{\theta_2/\theta_1} \in \mathcal{G}_{N_0}. $$
\end{enumerate}
\end{proposition}
Finally, note that every commuting family of random variables, $\mathcal{N}_{N_0}$, can be identified by its Koenigs function.
Furthermore, given that every element in $\mathcal{G}_{N_0}$ commutes with every other element in it,
the choice of the reference variable $N_0$ is arbitrary within the same family (excluding the identity $N_I$)
\begin{proposition}
The sets $\mathcal{G}_{N_0}$, $\mathcal{N}_{N_0}$, and the Koenigs function
$K_{N_0}$, are invariant to the choice of base random variable;
If $N \in \mathcal{N}_{N_0}$ is such that $N \neq N_I$, then
$\mathcal{N}_{N} = \mathcal{N}_{N_0}$, $\mathcal{G}_{N} = \mathcal{G}_{N_0}$, and $K_N = K_{N_0}$.
\label{prop:erpl1}
\end{proposition}

\subsection{Series expansion of the functions in $\mathcal{G}_{N_0}$ and $\mathcal{H}_{N_0}$}

The following theorem establishes a fundamental property of the commuting families, that is,
that the coefficients in the power series expansion of any $g_\theta \in \mathcal{G}_{N_0}$
are polynomials in the parameter $\theta$, with their degree bounded by the index of the term.
\begin{theorem}
Let $N_0$ be a positive integer-valued random variable with $0 < p_1(N_0) <1$. For any
$g_{\theta} \in \mathcal{G}_{N_0}$, its power series expansion around the origin is given by
$$g_{\theta}(t) = \theta t + \sum_{i=2}^{\infty}a_i(\theta)t^i, $$
where each coefficient $a_i(\theta)$ is a polynomial in $\theta=g_{\theta}^{\prime}(0)$ of degree at most $i$.
Specifically, for $i \ge 2$, we have
$a_i(\theta) = \theta(1-\theta)Q_i(\theta)$, where $Q_i(\theta)$ is a polynomial of degree at most $i-2$.
\label{thm:exp}
\end{theorem}

\begin{proof}
Let the Taylor expansion of the Koenigs function and its inverse be
$K_{N_0}(t) = t + \sum_{i=2}^{\infty} c_i t^i$, and $K_{N_0}^{-1}(t)=\sum_{i=1}^{\infty} b_i t^i$.
By Definition \ref{def:smn1},
$g_{\theta}(t)=K^{-1}_{N_0}(\theta K_{N_0}(t))$.

Taking the $n$-th derivative of $g_\theta$ with respect to $\theta$ using the chain rule yields:
$$ \frac{\partial^{n}g_{\theta}(t)}{\partial\theta^n}=\frac{\partial^n K^{-1}_{N_0}}{\partial t^n}(\theta K_{N_0}(t)) [K_{N_0}(t)]^n. $$
Evaluating this at $\theta=0$, we obtain:
$$ \left. \frac{\partial^{n}g_{\theta}(t)}{\partial\theta^n} \right|_{\theta=0} = \frac{\partial^nK^{-1}_{N_0}}{\partial t^n}(0) [K_{N_0}(t)]^n = n!b_{n}[K_{N_0}(t)]^n. $$
Summing the Taylor series of $g_\theta(t)$ in $\theta$ gives:
$$ g_{\theta}(t) =
\sum_{n=1}^{\infty} b_n [K_{N_0}(t)]^n \theta^n. $$
Now, consider the coefficient of $t^i$ in the expansion of $g_\theta(t)$, denoted $a_i(\theta)$.
Let $[K_{N_0}(t)^n]_i$ represent the $i$-th Taylor coefficient of the $n$-th power of the Koenigs function.
Then,
$$ a_i(\theta) = \sum_{n=1}^\infty b_n[K_{N_0}(t)^n]_i \theta^n. $$
Since $K_{N_0}(t) = t + \sum_{i=2}^{\infty} c_i t^i$, the lowest-order term in $[K_{N_0}(t)]^n$
is $t^n$. Therefore, $[K_{N_0}(t)^n]_i = 0$ for all $n>i$. This truncates the infinite sum into a polynomial
of degree $i$:
$$ a_i(\theta) = \sum_{n=1}^i b_n[K_{N_0}(t)^n]_{i} \theta^n. $$
Finally, observe that $g_1(t) = K^{-1}_{N_0}(K_{N_0}(t)) = t$, which implies $a_i(1) = 0$ for all $i \ge 2$.
Furthermore, $g_0(t) = K^{-1}_{N_0}(0) = 0$ implies $a_i(0) = 0$.
Consequently, $a_i(\theta)$ must contain the factor $\theta(1-\theta)$, completing the proof.
\end{proof}

Since $\mathcal{H}_{N_0} \subset \mathcal{G}_{N_0}$, this polynomial structure describes the probabilities of the stopping model.
\begin{corollary}
Let $N_0$ be a positive integer-valued random variable with $0<p_1(N_0)<1$. Then, for any
$N_{\theta} \in \mathcal{N}_{N_0}$, the probabilities $p_i(\theta) = \Pr(N_\theta=i)$
are polynomials in $\theta=p_1(\theta)$ of degree at most $i$. Specifically, for each $i \ge 2$ they take the form
$p_i(\theta) = \theta(1-\theta)Q_i(\theta)$.
\label{col:pol}
\end{corollary}

\subsection{Instances where $\mathcal{N}_{N_0}$ has a non-empty interior}

By Proposition \ref{prop:init2} and Theorem \ref{thm:par},
a necessary and sufficient condition for the set $\mathcal{G}_{N_0}$ to have a non-empty interior
is $0<p_1(N_0)<1$. Under this condition, the parameter
$\theta=g^{\prime}(0)$ spans the entire interval $(0,\infty)$.
While this condition is also necessary for $\mathcal{N}_{N_0}$ to have a non-empty interior (Proposition \ref{prop:zn2}),
it is not sufficient.
The following result establishes that if $\mathcal{N}_{N_0}$ has
a non-empty interior, its parameter space contains a maximal interval $(0,\theta_{N_0}]$.
\begin{theorem}
Let $N_0$ be a positive integer-valued random variable with $0<p_1(N_0)<1$.
If the parameter space $\mathcal{D}_{N_0}$ of the model $\mathcal{N}_{N_0}$ has a non-empty interior,
then there exists a threshold $\theta_{K_{N_0}} \in (0,1]$ such that: (i) $(0,\theta_{K_{N_0}}] \subset \mathcal{D}_{N_0}$, and
(ii) $(0,\theta_{K_{N_0}}+\epsilon] \not\subset \mathcal{D}_{N_0}$ for any $\epsilon>0$.
 \label{thm:intp}
\end{theorem}

\begin{proof}
Assume $\mathcal{D}_{N_0}$ has a non-empty interior. Then,
there exist an open interval $(\theta_1,\theta_2) \subset \mathcal{D}_{N_0}$.
Since $\mathcal{N}_{N_0}$ is closed under pgf composition,
for any $k \in \mathbb{N}^{+}$, if $\theta \in (\theta_1,\theta_2)$ then
the $k$-fold composition $N_{\theta}^{\circ k} = N_{\theta^k}$ must also lie in $\mathcal{N}_{N_0}$,
implying that $(\theta_1^k, \theta_2^k) \subset \mathcal{D}_{N_0}$ for every $k$.

As $k$ increases, these intervals approach the origin.
Intervals $(\theta_1^k, \theta_2^k)$ and $(\theta_1^{k+1}, \theta_2^{k+1})$ eventually overlap,
when the right endpoint of the $(k+1)$-th interval is greater than the left endpoint of the
$k$-th one, $\theta_2^{k+1} > \theta_1^k$, which happens for all $k > -\log(\theta_1)/\log(\theta_2/\theta_1)$.
Let $k_0$ be the smallest such integer.
The union of these overlapping intervals $\cup_{k=k_0}^{\infty} (\theta_1^k, \theta_2^k)$ forms a single connected
interval $(0,\theta_2^{k_0}) \subset \mathcal{D}_{N_0}$.
That establishes that $\mathcal{D}_{N_0}$ contains an
interval anchored at the origin.

We then define the threshold $\theta_{K_{N_0}}=\sup\{\theta  \mid (0,\theta] \subset \mathcal{D}_{N_0} \}$.
By construction, $\theta_{K_{N_0}} \ge \theta_2^{k_0} > 0$, and properties (i) and (ii) follow from the
definition of $\theta_{K_{N_0}}$ as a supremum.
Furthermore, note that $N_{\theta_{K_{N_0}}} \in \mathcal{N}_{N_0}$
since all coefficients $a_i$ of $h_{\theta_{K_{N_0}}}$ are non-negative. Indeed, if
$a_k(\theta_{K_{N_0}})<0$ for some index $k$, by continuity there would exist $\epsilon>0$ such that $a_k(\theta_{K_{N_0}}-\epsilon) < 0$,
which contradicts the fact that $\theta_{K_{N_0}}$ is the supremum.
\end{proof}

\begin{remark}
By Proposition \ref{prop:com}, $h_{N_{\theta_1}}^{-1}\circ h_{N_{\theta_2}} = g_{\theta_2/\theta_1} \in \mathcal{G}_{N_0}$.
If $\theta_1,\theta_2$ take values in $(0,\theta_{K_{N_0}})$, the ratio
$\theta_2/\theta_1$ covers the entire range $(0,\infty)$.
Thus, whenever $\mathcal{N}_{N_0}$ has a non-empty interior,
$$ \mathcal{G}_{N_0} =
\left\{ h_{N_{\theta_1}}^{-1} \circ h_{N_{\theta_2}} \mid h_{\theta_1}, h_{\theta_2} \in \mathcal{H}_{N_0}  \right\}, $$
and thus $\mathcal{G}_{N_0}$ can be fully reconstructed as the set of ratios of pgfs from $\mathcal{H}_{N_0}$.
\end{remark}

A consequence of Theorem \ref{thm:intp}, proved in Appendix 1,
is that if $N_0$ has finite support, then $\mathcal{N}_{N_0}$ has empty interior.
Hence, $0<p_1(N_0)<1$ is not sufficient to guarantee a non-empty interior.
\begin{proposition}
Let $N_0$ be a positive integer-valued random variable with pgf $h_{N_0} = \sum_{i=1}^k p_i t^i$ where $1<k<\infty$
and $p_1, p_k \in (0,1)$. Then, the interior of $\mathcal{N}_{N_0}$ is empty.
\label{prop:finit}
\end{proposition}
Since the identity $N_I = N_{\theta=1}$ is always in $\mathcal{N}_{N_0}$, any threshold
$\theta_{K_{N_0}} < 1$ would imply that the identity $\theta=1$ is
an isolated point, rendering the model disconnected. Hence,
$\mathcal{N}_{N_0}$ being connected requires $\theta_{K_{N_0}}=1$.
\begin{corollary}
The stopping model $\mathcal{N}_{N_0}$ is connected and has a non-empty interior
if and only if its parameter space for $\theta$ is $\mathcal{D}_{N_0} = (0,1]$.
\label{cor:pfc}
\end{corollary}

The threshold $\theta_{K_{N_0}}$ in Theorem \ref{thm:intp} is a fundamental constant
determined by the Koenigs function $K_{N_0}$, and
determining the parameter space of $\mathcal{N}_{N_0}$.
We conclude by describing all sub-models of $\mathcal{N}_{N_0}$ that are closed,
connected and with a non-empty interior. Each of these shares the same $K_{N_0}$, and
its parameter space is an interval anchored at $0$, with upper bound not
exceeding $\theta_{K_{N_0}}$.

\begin{corollary}
Let $N_0$ be a positive integer-valued random variable with $0<p_1(N_0)<1$.
Any sub-model $\mathcal{N} \subset \mathcal{N}_{N_0}$ that is closed under pgf composition,
connected and with a non-empty interior must be such that its interior set is
$$  \mathring{\mathcal{N}} =
\left\{ N_{\theta} \in \mathcal{N}_{N_0} \mid \theta \in (0,\theta_\mathcal{N})  \right\},     $$
for some $0 < \theta_\mathcal{N} \le \theta_{K_{N_0}}$.
\label{cor:gdmd}
\end{corollary}
From now on, we will recognize $(0, \theta_{K_{N_0}}]$ as the maximal parameter space associated with
the Koenigs function $K_{N_0}$.

In the following section, we find that finite-dimensional positive stopping models closed under pgf composition, connected
and with a non-empty interior are necessarily as in Corollary \ref{cor:gdmd}.

\section{Composition-closed and connected finite-dimensional models }

The stopping models $\mathcal{N}_{N_0}$, defined to be the set of all random variables that commute with a given $N_0$, are
closed under pgf composition and have a parameter space of at most one dimension.
Conversely, any positive stopping model that is closed under pgf composition, connected and has a non-empty interior,
yet is not a subset of a $\mathcal{N}_{N_0}$ for some $N_0$,
must necessarily contain at least two non-commuting random variables.

In this section, we demonstrate that if such a stopping model
contains non-commuting random variables, it must be infinite-dimensional.
It follows that any finite-dimensional, positive stopping model satisfying these topological and algebraic properties
must consist of a family of random variables that commute, thereby restricting its parameter space to be one-dimensional.

\subsection{Preliminary notation and result}

This section investigates finite-dimensional, positive stopping models $\mathcal{N}$
that are closed under pgf composition, connected, and possess a non-empty interior.
We assume $\mathcal{N}$ is parametric, though it might only admit a local parametrization;
consequently, our notation in this section avoids reliance on a global parameter space.
We further assume that the probabilities $p_i(N) = \Pr(N=i)$ are continuously differentiable
with respect to the parameters for all $i \ge 1$.

We partition $\mathcal{N}$ into submodels $\mathcal{N}_s$ $(s \ge 1)$
based on the smallest index with a non-zero probability:
$$ \mathcal{N}_s = \left\{ N \in \mathcal{N} \mid p_i(N)=0 \text{ for } i<s, \text{ and } p_s(N)>0  \right\} \subset \mathcal{N}. $$

The order of a stopping model $\mathcal{N}$, denoted by $o$, is the smallest integer $s$
for which $\mathcal{N}_s$ is non-empty:
$$  o=\min\{ s \mid \mathcal{N}_{s} \neq \emptyset \}.  $$
Under these definitions, the following properties hold:
\begin{enumerate}
\item
If $N_{1} \in \mathcal{N}_{s_1}$ and $N_{2} \in \mathcal{N}_{s_2}$, then
$N_{1}\circ N_{2} \in \mathcal{N}_{s_1 s_2}$. In particular, the k-fold composition
$N_{1}^{\circ k}$ lies in $\mathcal{N}_{s_1^k}$.
\item
If $o=1$, there exists at least one $N \in \mathcal{N}$ such that $p_1(N)>0$. If $\mathcal{N}$ is closed
under pgf composition, then the submodel $\mathcal{N}_1$ is likewise closed.
In certain cases, $\mathcal{N}_{1}=\mathcal{N}$.
\item
If $o>1$, then $p_1(N)=0$ for all $N \in \mathcal{N}$. Here,
the submodel $\mathcal{N}_o$ cannot be closed under pgf composition because
$N_{1}, N_{2} \in \mathcal{N}_o$ implies $N_1\circ N_2 \in \mathcal{N}_{o^2}$,
where $\mathcal{N}_o\cap\mathcal{N}_{o^2}=\emptyset$. Note that
$\mathcal{N}_{o^k}\neq\emptyset$ for all $k \ge 1$.
\item
The submodel $\mathcal{N}_o = \mathcal{N} \cap \{ N \mid p_o(N) > 0 \}$ is an open subset of $\mathcal{N}$
and therefore shares the same dimensionality as the full model $\mathcal{N}$.
\end{enumerate}
The following proposition (proven in Appendix 2) demonstrates that for a finite-dimensional, connected, positive stopping
model to be closed under composition and possess a non-empty interior, it must
necessarily satisfy $p_1(N)>0$ for all $N \in \mathcal{N}$.
\begin{proposition}
Let $\mathcal{N}$ be a finite-dimensional, positive stopping model that is closed under pgf composition
and has a non-empty interior. Assume that the
probabilities $p_i(N)$ are continuously differentiable with respect to the parameters for all $i$.
Then, $\mathcal{N}$ is connected only if $\mathcal{N} = \mathcal{N}_1$.
 \label{prop:tpr1}
\end{proposition}

\subsection{Main result}

The central result of this paper establishes that any finite-dimensional positive stopping model
closed under pgf composition, connected and possessing a non-empty interior must consist of
a family of random variables that commute. Consequently, such a model is necessarily one-dimensional.
\begin{theorem}
Let $\mathcal{N}$ be a finite-dimensional, positive stopping model that is
closed under pgf composition, is connected, and has a non-empty interior.
Assume that the probabilities $p_i(N)=\Pr(N=i)$ are continuously differentiable with respect to the parameters for all $i$.
Then, $\mathcal{N}$ is a family of random variables that commute taking the form described in Corollary \ref{cor:gdmd}.
\label{thm:ult2}
\end{theorem}
The proof of Theorem \ref{thm:ult2} rests on two key propositions (proven in Appendix 2).
The first demonstrates that for a model $\mathcal{N}$ satisfying the conditions of the theorem,
the intersection of its associated set of pgfs, $\mathcal{H}_\mathcal{N}$,
with the set of pgfs that commute with $h_{N_0}$ for $N_0 \in \mathcal{N}$, $\mathcal{H}_{N_0}$,
retains a non-empty interior (relative to $\mathcal{H}_{N_0}$).
\begin{proposition}
Let $\mathcal{N}$ satisfy the conditions of Theorem \ref{thm:ult2}.
Then, there exists an element $N_0 \in \mathring{\mathcal{N}}$, a constant $\epsilon > 0$, and a
differentiable curve
$$ \mathcal{R}_0 = \left\{\mathcal{R}_0(r) = N_r \mid r \in (1-\epsilon,1+\epsilon) \right\} \subset \mathcal{N}, $$
such that $\mathcal{R}_0(1)=N_0$ and $\mathcal{R}_0 \subset \mathcal{N}_{N_0}$.
Consequently, $\mathcal{R}_0 \subset \mathcal{N} \cap \mathcal{N}_{N_0}$.
 \label{prop:ff1}
\end{proposition}
The second proposition
provides a ``dimension-explosion" argument; if the stopping model contains even two non-commuting random variables,
the dimension of $\mathcal{H}_\mathcal{N}$ must become arbitrarily large,
contradicting the initial finite-dimensionality assumption.
\begin{proposition}
Let $\mathcal{N}$ satisfy the conditions of Theorem \ref{thm:ult2},
and let $h_{N_0} \in \mathring{\mathcal{H}}_{\mathcal{N}}$. If there exists a pgf $h_{N} \in \mathcal{H}_\mathcal{N}$
that does not commute with $h_{N_0}$ and $n \in \mathbb{N}^{+}$, then the set of compositions:
$$ \mathcal{G}_{N_0,N,n} = \left\{ g_{\theta_1,\ldots,\theta_n} =
g_{\theta_1} \circ h_{N} \circ g_{\theta_2} \circ h_N
\circ \dots \circ  g_{\theta_n} \circ h_N \mid g_{\theta_i} \in \mathcal{G}_{N_0}   \right\},   $$
is such that the dimension of $\mathcal{G}_{N_0,N,n}$ grows unboundedly as $n \to \infty$.
\label{prop:ff2}
\end{proposition}

\begin{proof}[Proof of Theorem $\ref{thm:ult2}$]
Let $N_0 \in \mathcal{\mathring{N}}_1$, which is non-empty by Proposition $\ref{prop:tpr1}$.
According to Proposition $\ref{prop:ff1}$, there exists an open subset of $\mathcal{G}_{N_0}$ defined by
$\{ h_{\theta} \in \mathcal{H}_{N_0} \mid \theta \in (\theta_0-\epsilon,\theta_0+\epsilon)\ \}$
that is contained within $\mathcal{H}_\mathcal{N}$.

We proceed by contradiction to show that every element of $\mathcal{H}_\mathcal{N}$ must commute with $h_{N_{0}}$.
Suppose there exists some $h_{N} \in \mathcal{H}_\mathcal{N}$ that does not commute with $h_{N_0}$.
For each $n \in \mathbb{N}^{+}$, define the set of compositions:
$$ \mathcal{H}_{N_{0},N,n} = \left\{ h_{\theta_1,\ldots,\theta_n} = h_{\theta_1} \circ h_{N} \circ h_{\theta_2} \circ h_{N}
\circ \dots \circ h_{\theta_n} \circ h_{N} \mid \theta_i \in (\theta_0-\epsilon,\theta_0+\epsilon)  \right\}.   $$
By construction, $\mathcal{H}_{N_{0},N,n}$ is an open subset of both $\mathcal{H}_\mathcal{N}$ and the set
$\mathcal{G}_{N_{0},N,n}$ introduced in Proposition \ref{prop:ff2}.
This proposition implies that the dimension of this open subset grows unboundedly as $n \to \infty$.
Since this construction holds for all $n \in \mathbb{N}^{+}$, the presence of both
$N_0$ and $N$ in $\mathcal{N}$ forces $\mathcal{H}_\mathcal{N}$ (and thus $\mathcal{N}$) to be infinite-dimensional,
contradicting our finite-dimensionality assumption.

Consequently, every element in $\mathcal{H}_\mathcal{N}$ must commute with $h_{N_0}$, implying
$\mathcal{N} \subset \mathcal{N}_{N_0}$. This allows the model $\mathcal{N}$ to be parametrized globally by
$\theta = p_1(N)$. Given that $\mathcal{N}$
is closed under pgf composition, connected, and has a non-empty interior, the model has to be as in
Corollary $\ref{cor:gdmd}$.
\end{proof}

\subsection{Properties of composition-closed and connected finite-dimensional models}

Theorem \ref{thm:ult2} establishes that any finite-dimensional positive stopping model
closed under pgf composition, connected, and possessing a non-empty interior must be
a family of commuting random variables supported on the positive integers, all with a non-zero probability of equaling one.

The following result summarizes the structural core properties of these models. These properties follow
from this characterization and the results for commuting models established in Section 3.
\begin{proposition}
Let $\mathcal{N}$ be a finite-dimensional, positive stopping model that is closed under pgf composition, is connected,
and has a non-empty interior. If the probabilities
$p_i(N)=\Pr(N=i)$ are continuously differentiable with respect to the parameters for all $i$, then:
\begin{enumerate}
\item{}
$p_1(N) > 0$ for all $N \in \mathcal{N}$.
\item{}
The model $\mathcal{N}$ admits a global, one-dimensional parametrization $\theta=h_{N}^\prime(0)=p_1(N)$.
\item{}
The parameter space is either $\mathcal{D}_\mathcal{N}=(0,\theta_{\mathcal{N}}]$ or $\mathcal{D}_\mathcal{N}=(0,\theta_{\mathcal{N}})$,
with $0 < \theta_{\mathcal{N}} \le \theta_{K_{\mathcal{N}}} \le 1$, where $\theta_{K_{\mathcal{N}}}$ is the threshold value in Theorem \ref{thm:intp}
that determines the maximal parameter space associated with this family of commuting random variables, namely, $(0,\theta_{K_{\mathcal{N}}}]$.
\item{}
If $N_I=N_{\theta=1} \in \mathcal{N}$, then $\theta_{K_{\mathcal{N}}}=1$, and the parameter space is $\mathcal{D}_\mathcal{N}=(0,1]$.
\item{}
For any $N_{\theta_1}, N_{\theta_2} \in \mathcal{N}$, their pgfs commute:
$$ h_{N_{\theta_1}}\circ h_{N_{\theta_2}} =
h_{N_{\theta_2}}\circ h_{N_{\theta_1}} = h_{N_{\theta_1 \theta_2}} \in \mathcal{H}_\mathcal{N}.  $$
\item{}
For $i \ge 2$, the probability $p_i(N_\theta) = \theta(1-\theta)Q_i(\theta)$ is a polynomial in $\theta$
of degree at most $i$. In particular, $p_2(N_\theta) \propto \theta(1-\theta)$.
\item{}
The set of pgfs of the random variables in $\mathcal{N}$ is characterized by the Koenigs function $K_{\mathcal{N}}(t)$:
$$ \mathcal{H}_\mathcal{N} = \left\{h_{N_\theta}(t) = K_{\mathcal{N}}^{-1}(\theta  K_{\mathcal{N}}(t)) \mid
\theta = h_{N_\theta}^\prime(0) \in \mathcal{D}_\mathcal{N}  \right\}, $$
where $K_{\mathcal{N}}(t)$ is given for any $N_\theta \in \mathcal{N}$ by
$$ K_{\mathcal{N}}(t) = \lim_{n\rightarrow\infty}\frac{h_{N_\theta}^{\circ n}(t)}{h_{N_\theta}^{\prime}(0)^n} = \left. \frac{d h_{N_\theta}(t)}{d \theta} \right|_{\theta=0}. $$
Likewise, the set $\mathcal{G}_\mathcal{N}$ of non-constant analytic functions commuting with the pgfs in $\mathcal{H}_\mathcal{N}$ is obtained
by extending the parameter space of $\mathcal{H}_\mathcal{N}$ to $(0,\infty)$:
$$ \mathcal{G}_\mathcal{N} = \left\{g_{\theta}(t) = K_{\mathcal{N}}^{-1}(\theta  K_{\mathcal{N}}(t)) \mid
\theta=g_{\theta}^\prime(0) \in (0,\infty)  \right\}. $$
\item{}
The Koenigs function $K_{\mathcal{N}}(t)$ determines the threshold value $\theta_{K_{\mathcal{N}}}$, as well as the maximal parameter space $(0, \theta_{K_{\mathcal{N}}}]$,
associated with this family of commuting random variables.
\item{}
For any $N_{\theta} \in \mathcal{N}$, then
$$   h_{N_{\theta}}^{-1} = g_{1/\theta} \in \mathcal{G}_\mathcal{N}.   $$
\item{}
For any $N_{\theta_1}, N_{\theta_2} \in \mathcal{N}$, their pgfs satisfy
$$ h_{N_{\theta_1}}^{-1}\circ h_{N_{\theta_2}} =
h_{N_{\theta_2}}\circ h_{N_{\theta_1}}^{-1} = g_{\theta=\theta_2/\theta_1} \in \mathcal{G}_\mathcal{N},  $$
where $\theta = \theta_2/\theta_1 \in (0,\infty)$. Consequently,
$$  \mathcal{G}_\mathcal{N} = \left\{ h_{N_{\theta_1}} \circ h_{N_{\theta_2}}^{-1} \mid h_{N_{\theta_1}}, h_{N_{\theta_2}} \in \mathcal{H}_\mathcal{N} \right\}.  $$
\end{enumerate}
\label{prop:prpts}
\end{proposition}
In the following section, we use the function $K_{\mathcal{N}}(t)$ to characterize stopping models
that are closed and have maximal parameter space $(0,1]$, including the identity $N_I$.

\section{Composition-closed and connected models that include $N_I$}

By Proposition \ref{prop:stst}, randomly stopped sum and extreme transformations (Definitions \ref{def:rssm} and \ref{def:rsex})
yield statistically stable model extensions provided the underlying stopping model is closed under pgf composition
and contains the identity $N_I=N_{\theta=1}$.

In this section, we characterize the subclass of stopping models that satisfy these two requirements while also meeting the regularity
conditions introduced in Section 4. These models must take the form described in Proposition \ref{prop:prpts}, with
parameter space $\mathcal{D}_\mathcal{N}=(0,1]$.

To achieve this, we introduce a function $\varphi_{\mathcal{N}}(t)$, whose analytical properties are
established in the following result, proven in Appendix 3.

\begin{proposition}\label{prop:vphi1}
Let $\mathcal{N}$ be a finite-dimensional, positive stopping model that is closed under pgf composition, is connected and has
a non-empty interior, and let $p_i(N)$ be continuously differentiable with respect to the parameters for all $i$. Define
$$ \varphi_{\mathcal{N}}(t) = 1 - \frac{K_{\mathcal{N}}(t)}{t K_{\mathcal{N}}^{\prime}(t)},    $$
and let $m$ denote the multiplicity of the vertical asymptote of $K_{\mathcal{N}}(t)$ at $t=1$,
$$   m = \lim_{t\rightarrow 1^{-}} \frac{\log K_{\mathcal{N}}(t)}{\log (1-t)}.   $$
Then:
(i) $\varphi_{\mathcal{N}}(t)$ is analytic in $[0,1)$,
(ii) $\varphi_{\mathcal{N}}(0)=0$ and $\lim_{t\rightarrow 1^{-}}\varphi_{\mathcal{N}}(t)=1$, and
(iii) $ \varphi_{\mathcal{N}}^\prime(1) =\lim_{t\rightarrow 1^{-}}\varphi_{\mathcal{N}}^\prime(t) = - \frac{1}{m}$, where $m \in [-1,0]$.
Furthermore, $K_{\mathcal{N}}(t)$ is uniquely determined by $\varphi_{\mathcal{N}}(t)$ through the relation:
$$ K_{\mathcal{N}}(t) = t \exp \left(\int_{0}^{t} \frac{\varphi_{\mathcal{N}}(s)}{s(1-\varphi_{\mathcal{N}}(s))} ds \right). $$
\end{proposition}
The following theorem provides necessary and sufficient conditions for
a composition-closed and connected stopping model to include $N_I$, and therefore to admit the parameter space $(0,1]$ for $\theta$.
\begin{theorem}
Let $\mathcal{N}$ be a finite-dimensional, positive stopping model that is closed under pgf composition, is connected, and has
a non-empty interior. Assume that the probabilities $p_i(N)=\Pr(N=i)$ are continuously differentiable with respect to the
parameters for all $i$. Under the parametrization $\theta=p_1(N)$, the model admits the maximal parameter space $\mathcal{D}_\mathcal{N}=(0,1]$
if and only if $\varphi_{\mathcal{N}}(t)$ is itself the pgf of a random variable satisfying $\varphi_{\mathcal{N}}(0)=0$.

Consequently, there exists a bijective correspondence (duality) between the set of probability distributions
supported on the positive integers and the class of positive stopping models
that are closed under pgf composition, have continuously differentiable probabilities, and have parameter space $(0,1]$.
 \label{thm:nnn}
\end{theorem}

\begin{proof}
By Theorem \ref{thm:ult2}, the stopping model $\mathcal{N}$ must consist of a family of random variables that commute,
with parameter space $\mathcal{D}_{\mathcal{N}} = (0,1]$  under the $\theta$ parametrization. For this to hold, each function
$g_{\theta} = K_{\mathcal{N}}^{-1}(\theta  K_{\mathcal{N}}(t)) \in \mathcal{G}_{\mathcal{N}}$ with $\theta \in (0,1]$
must be a valid pgf.

Define the function:
$$ K_{\theta}(t) = g_{\theta}(t)/\theta = t + \sum_{n=2}^\infty c_n(\theta) t^n. $$
Using L'Hopital's rule and the fact that ${K_{\mathcal{N}}^{-1}}^{\prime}(0)=1/K_{\mathcal{N}}^\prime(0)=1$, we observe that:
$$ \lim_{\theta\rightarrow0+}K_{\theta}(t) =
\lim_{\theta\rightarrow0+} \frac{K_{\mathcal{N}}^{-1}(\theta  K_{\mathcal{N}}(t))}{\theta} =
\lim_{\theta\rightarrow0+} {K_{\mathcal{N}}^{-1}}^{\prime}(\theta  K_{\mathcal{N}}(t)) K_{\mathcal{N}}(t)  = K_{\mathcal{N}}(t), $$
Moreover, $\lim_{\theta\rightarrow1-}K_{\theta}(t) = t,$ implying that $c_n(1)=0$ for all $n \ge 2$.
Since $K_{\mathcal{N}}(g_{\theta}(t))= \theta K_{\mathcal{N}}(t)$, differentiating with respect to $\theta$ yields:
$$ K_{\mathcal{N}}^{\prime}(g_{\theta}(t)) \frac{\partial}{\partial\theta}g_{\theta}(t) = K_{\mathcal{N}}(t) =
\frac{K_{\mathcal{N}}(g_{\theta}(t))}{\theta}.  $$
Rearranging for the partial derivative, we find:
$$ \frac{\partial}{\partial\theta}g_{\theta}(t) = \frac{K_{\mathcal{N}}(g_{\theta}(t))}{\theta K_{\mathcal{N}}^{\prime}(g_{\theta}(t))}.   $$
Substituting this into the derivative of $K_{\theta}(t)$ gives:
$$ \frac{\partial}{\partial\theta}K_{\theta}(t) =
\sum_{n=2}^\infty \frac{\partial c_n(\theta)}{\partial\theta} t^n =
\frac{\frac{\partial}{\partial\theta}g_{\theta}(t)}{\theta} - \frac{g_{\theta}(t)}{\theta^2} = $$
$$ \frac{1}{\theta^2}\left(\frac{K_{\mathcal{N}}(t)}{K_{\mathcal{N}}^{\prime}(t)}-t\right)\circ g_{\theta}(t) =
- \frac{g_\theta(t) \varphi_{\mathcal{N}}(g_\theta(t))}{\theta^2}.  $$
Since $g_{\theta=1}(t) = t$,
by Theorem \ref{thm:exp} the coefficients $c_n(\theta)$ for $n>1$ are polynomials of degree $n-1$
satisfying $c_n(1)=0$. If the coefficients
$c_n(\theta)$ are decreasing in $\theta$ on $(0,1]$, then $c_n(1)=0$ implies $c_n(\theta) \ge 0$ on this interval,
ensuring $g_{\theta}\left(t\right)$ is a valid pgf for all $\theta \in (0,1)$.

Thus, the proof reduces to showing that $c_n(\theta)$
is decreasing for all $n$ if and only if all the coefficients $d_i$ in the power series expansion of:
$$ \varphi_{\mathcal{N}}(t) = 1 - \frac{K_{\mathcal{N}}(t)}{t K_{\mathcal{N}}^{\prime}(t)} = \sum_{i=1}^\infty d_i t^i    $$
are non-negative and so that $\varphi_{\mathcal{N}}(t)$ is a valid pgf with $\varphi_{\mathcal{N}}(0)=0$.

Forward implication: Assume $c_n(\theta) \ge 0$ for all $n$ and $\theta \in (0,1)$.
Since $c_n(1)=0$, continuity implies that:
$$ \frac{\partial c_n(\theta)}{\partial\theta}_{|\theta=1} \le 0. $$
From our derivative expression,
$$ \frac{\partial c_n(\theta)}{\partial\theta}_{|\theta=1}  =
\left[\frac{K_{\mathcal{N}}(t)}{K_{\mathcal{N}}^{\prime}(t)}-t\right]_n = \left[ - t \varphi_{\mathcal{N}}(t) \right]_n = -d_{n-1}. $$
Therefore, $-d_{n-1} \le 0$ for all $n > 1$ implies $d_{i} \ge 0$ for all $i \ge 1$.

Converse implication: Assume that $d_i \ge 0$ for all $i \ge 1$.
We proceed by induction on $n$.
The base case is trivial as $c_1(\theta) = 1 > 0$.
Assume that
$c_k(\theta) \ge 0$ for all $k < n$.
The $n$-th coefficient of the derivative is:
$$ \frac{\partial c_n(\theta)}{\partial\theta} = -\frac{1}{\theta^2}\left[\left(t-\frac{K_{\mathcal{N}}(t)}{K_{\mathcal{N}}^{\prime}(t)}\right)\circ g_{\theta}(t)\right]_n =
-\frac{1}{\theta^2}\left[\sum_{i=2}^\infty d_{i-1} g_\theta(t)^i \right]_n.   $$
The coefficient of $t^n$ in the expansion of $g_\theta(t)^i$ for $i \ge 2$
depends only on the coefficients $c_k(\theta)$ for $k<n$; the first term in which $c_n(\theta)$ could appear is the
contribution of $d_1 g_\theta(t)^2$ producing the term $2 d_1 c_1(\theta) c_n(\theta) t^{n+1}$,
which does not affect the coefficient of $t^n$.

Since all $d_i$ and all $c_k(\theta) \ge 0$ for $k < n$ are non-negative because of the induction hypotheses,
the entire bracketed expression is non-negative. Therefore:
$$ \frac{\partial c_n(\theta)}{\partial\theta} =
\frac{-1}{\theta^2}\left[\sum_{i=2}^\infty d_{i-1} g_\theta(t)^i \right]_n \le 0,   $$
confirming that $c_n(\theta)$ is decreasing on $(0,1)$. Continuity and $c_n(1)=0$ then
imply $c_n(\theta) \ge 0$ for all $\theta \in (0,1)$.

Finally, to establish the duality between these models and positive integer-valued distributions, note that
$h(t) = \varphi_{\mathcal{N}}(t)$ satisfies the conditions to be such a pgf (see Remark 1).
Conversely, let $h(t)$ be the such a pgf. Proposition \ref{prop:vphi1} provides a way to construct a stopping model $\mathcal{N}_h$
satisfying our conditions. This model is the family of random variables whose pgfs belong to the set
$\mathcal{H}_{\mathcal{N}_h}=\{ h_{N_\theta}(t)=K_{\mathcal{N}_h}^{-1}(\theta K_{\mathcal{N}_h}(t))|\theta\in(0,1] \},$ where:
$$  K_{\mathcal{N}_h}(t) = t \exp \left(\int_{0}^{t} \frac{h(s)}{s(1-h(s))} ds \right).  $$
\end{proof}

Theorem \ref{thm:nnn} implies that for a stopping model $\mathcal{N}$
to admit $(0,1]$ as its parameter space, the expected value of the random variable associated with
$\varphi_{\mathcal{N}}(t)$, given by $\varphi_{\mathcal{N}}^\prime(1)=-1/m$, must be not smaller than one.
This observation leads to a necessary condition on the index $m$:
\begin{corollary}
For any stopping model $\mathcal{N}$ satisfying the conditions of Theorem \ref{thm:nnn}, the index $m$ must satisfy:
$$   m = \lim_{t\rightarrow 1^{-}} \frac{\log K_{\mathcal{N}}(t)}{\log (1-t)} \in [-1,0].   $$
The limiting case $m=-1$ corresponds uniquely to the geometric model, where $\varphi_{\mathcal{N}}(t)=t$.
\end{corollary}

\begin{remark}
Theorem \ref{thm:nnn}, characterizing composition-closed positive stopping models with maximal parameter space $(0,1]$,
reformulates classical results by Karlin and McGregor (1968a, 1968b) on embedding discrete-time branching processes
into continuous time. Imposing this closure requires that for any $N_\theta \in \mathcal{N}$ and positive integer $n$,
there exists an $N_n \in \mathcal{N}$ such that $h_{N_\theta}=h_{N_n}^{\circ n}$. By Proposition \ref{prop:prpts},
$p_1(N_n)=\theta^{1/n}$ remains in $(0,1]$, ensuring $h_{N_n}$ is always a valid pgf. Therefore, every pgf in the
family must be infinitely divisible under composition, which is the algebraic requirement for
continuous-time embedding. Since a continuous process evolves via a pgf semigroup ($F_{t+s} = F_t \circ F_s$),
embedding $h_{N_\theta}$ at $t=1$ strictly requires every fractional time-slice $F_{1/n}$ to also be a valid pgf.
By guaranteeing these $n$-fold compositional roots exist and remain valid distributions across $(0,1]$,
Theorem \ref{thm:nnn} translates Karlin and McGregor's constraints into the framework of composition-closed stopping models.
\end{remark}

Example 5 in Section 6.1 illustrates how one takes advantage of the duality established in Theorem \ref{thm:nnn},
to construct a model closed under pgf composition with maximal parameter space $(0,1]$
starting from any arbitrary distribution on the positive integers. Note that the bijection is between positive-integer
distributions and maximal $(0,1]$ families of commuting random variables (and not with arbitrary submodels of those maximal families).

\section{Examples of composition-closed, connected stopping models}

Theorem \ref{thm:ult2} establishes that finite-dimensional positive stopping models
that are closed under pgf composition,
connected and with a non-empty interior must consist of families of commuting
positive integer-valued random variables that include one in their support.
Such families are necessarily uniparametric. Depending on their Koenigs function,
the parameter space may be connected or disconnected.

In this section we present examples of stopping models closed and connected. We distinguish the
ones that admit a maximal parameter space including $N_I$ from the ones that do not.

\subsection{Composition-closed, connected stopping models that include $N_I$}

This subsection presents stopping models and families of stopping models
whose Koenigs functions satisfy the conditions of Theorem \ref{thm:nnn}, thus admitting
a maximal parameter space $\theta \in (0,1]$.

By Proposition \ref{prop:stst}, all stopping models presented here
make the randomly stopped sum and extreme transformations in Definitions \ref{def:rssm}, \ref{def:rsex} and \ref{def:rscb},
into statistically stable (idempotent) extensions. In other words, applying any of these transformations using any of these stopping models
twice in succession to any statistical model leaves the transformed model unchanged.
These examples counter the assertion in Marshall and Olkin (1997) that,
under certain unspecified regularity conditions, the only stopping model closed under pgf composition is geometric (Example 2).

\begin{example}
The family of random variables given by
$$ \mathcal{N} = \left\{ N_\theta \mid h_{N_\theta}(t) = 1 - (1 - t)^\theta, \quad \theta \in (0,1]  \right\}, $$
is a stopping model closed under pgf composition that satisfies $E[N_\theta]=\infty$ for $\theta \in (0,1)$, and is recognized as the Sibuya model.
The Koenigs function of the model is:
$$ K_{\mathcal{N}}(t) = - \log{(1-t)}.  $$
Here, $m=0$,
$$ \varphi_{\mathcal{N}}(t) = 1 + \frac{1}{t} (1-t) \log (1-t),  $$
and the probability mass function is given by:
$$ \Pr(N_\theta = n) =
\frac{\theta(1-\theta)(2-\theta)\cdots(n-1-\theta)}{n!} \quad \mbox{for} \quad n \in \mathbb{N}^{+}.  $$
\end{example}

\begin{example}
The zero truncated geometric model:
$$ \mathcal{N} = \left\{ N_\theta \mid h_{N_\theta}(t) = \frac{\theta t}{1-(1-\theta) t}, \quad \theta \in (0,1] \right\},  $$
is closed under pgf composition, with Koenigs function:
$$ K_{\mathcal{N}}(t) =
\frac{t}{1-t}, $$
corresponding to the unique case where $m=-1$, and $\varphi_{\mathcal{N}}(t) = t$.
The probability mass function is $\Pr(N_\theta = n) = \theta (1-\theta)^{n-1}$ for $n \in \mathbb{N}^{+}$.
\end{example}

\begin{example}
For $\alpha \in (0,1)$, the family $\mathcal{N}_\alpha$ given by:
$$ \mathcal{N}_\alpha = \left\{ N_\theta^\alpha \mid h_{N_\theta^\alpha}(t) =
1 - \frac{1-t}{(\theta+(1-\theta)(1-t)^\alpha)^{1/\alpha}}, \quad \theta \in (0,1]  \right\}, $$
forms a stopping model closed under pgf composition with $E[N_\theta^\alpha]=\theta^{-1/\alpha}$
and $V[N_\theta^\alpha]=\infty$ for $\theta \in (0,1)$.
The Koenigs function is:
$$ K_{\mathcal{N}_\alpha}(t) = \frac{1}{\alpha} \left(\frac{1}{(1-t)^\alpha} - 1\right) =
t + \sum_{k=2}^{\infty} \frac{\prod_{i=1}^{k-1}(\alpha+i)}{k!} t^k.   $$
In this case $m=-\alpha$ and
$$  \varphi_{\mathcal{N}_\alpha}(t) = 1 - \frac{1-(1-t)^{\alpha}}{\alpha t}(1-t),   $$
which is a pgf if and only if $\alpha \in [0,1]$. The probability mass function of these stopping models is given by:
$$ \Pr(N_\theta^\alpha = n) =
\frac{(-1)^{n+1}}{n! \, \Gamma\left(\frac{1}{\alpha}\right)} \sum_{j=1}^n \Gamma\left(j + \frac{1}{\alpha}\right) (1-\theta)^j \sum_{k=1}^j \frac{(-1)^k \, \Gamma(\alpha k + 2)}{k! \, (j-k)! \, \Gamma(\alpha k + 2 - n)},   $$
for $n \in \mathbb{N}^{+}$. Examples 1 and 2 arise as limiting cases when $\alpha \rightarrow 0$ and $\alpha \rightarrow 1$ respectively.
\end{example}

\begin{example}
For $\alpha \in (0,1)$, consider the Koenigs function:
$$  K_{\mathcal{N}_\alpha}(t) = \frac{t}{(1-t)^\alpha} = \sum_{n=1}^{\infty} \frac{\Gamma(\alpha+n-1)}{(n-1)! \Gamma(\alpha)} t^n.  $$
For each $\alpha$, the family of random variables given by:
$$ \mathcal{N}_\alpha = \left\{ N_\theta^\alpha \mid h_{N_\theta^\alpha}(t) =
K_{\mathcal{N}_\alpha}^{-1}(\theta  K_{\mathcal{N}_\alpha}(t)), \quad \theta \in (0,1]  \right\}, $$
is a stopping model closed under pgf composition with finite moments.
Although the inverse of the Koenigs function lacks a simple
closed form expression, $h_{N_\theta^\alpha}$ can be evaluated numerically via the series representation:
$$ K_{\mathcal{N}_{\alpha}}^{-1}(t) = \sum_{n=1}^{\infty}\frac{(-1)^{n+1} \alpha \Gamma(\alpha n)}{(n-1)!\Gamma(\alpha n-n+2)} t^n. $$
Here $m=-\alpha$ and
$$  \varphi_{\mathcal{N}_\alpha}(t) = \frac{\alpha t}{1-(1-\alpha)t},  $$
which is a pgf if and only if $\alpha \in (0,1]$. The probability mass function for these stopping models
is given by:
$$ \Pr(N_\theta^\alpha = n) =
\alpha \sum^{n}_{j=1}\frac{(-1)^{j+1}\,\Gamma\left((\alpha-1)\,j+n\right)\,\theta^j}{(j-1)!\,\Gamma\left((\alpha-1)\,j+2\right)\,(n-j)!},  $$
for $n \in \mathbb{N}^{+}$. Example 2 arises as a limiting case of this example when $\alpha=1$.
\end{example}

\begin{example}
We now illustrate how the duality established in Theorem \ref{thm:nnn} can be used to
construct a stopping model closed under pgf composition with a maximal parameter
space of $(0,1]$ starting from an arbitrary distribution on the positive integers.

Consider the positive integer-valued distribution with pgf $h(t)=(t+t^2)/2$.
Setting $\varphi_{\mathcal{N}}(t)=h(t)$ and applying the recovery formula for $K_{\mathcal{N}}(t)$
from Proposition \ref{prop:vphi1}, we obtain:
$$ K_{\mathcal{N}_h}(t) = \frac{\sqrt[3]{2} t}{(2+t)^{1/3} (1-t)^{2/3}}.   $$
By Theorem \ref{thm:nnn}, the family of commuting random variables characterized by this Koenigs function,
$$ \mathcal{N}_h = \left\{ N_\theta^h \mid h_{N_\theta^h}(t) =
K_{\mathcal{N}_h}^{-1}(\theta  K_{\mathcal{N}_h}(t)), \quad \theta \in (0,1]  \right\}, $$
constitutes a stopping model closed under pgf composition. Although a closed formed expression for $h_{N_\theta^h}(t)$ exists,
it is analytically cumbersome. The probability mass function can be derived recursively via the Lagrange Inversion Theorem.
For instance, one can verify that:
$$  \Pr(N_\theta^h=2) = \Pr(N_\theta^h=3)  = \frac{1}{2}\theta(1-\theta),  $$
$$  \Pr(N_\theta^h=4) = \frac{5}{24}\theta (1-\theta)^2 (\theta+2). $$
\end{example}

\subsection{Composition-closed, connected stopping models that cannot include $N_I$}

The following families of models, first introduced and motivated in Valero and Ginebra (2025),
provide examples of stopping models closed under pgf composition with
a maximal parameter space $(0,\theta_{K_{\mathcal{N}}}]$ for some $0 < \theta_{K_{\mathcal{N}}} < 1$,
where $\theta_{K_{\mathcal{N}}}$ is the threshold value in Theorem \ref{thm:intp}.
These families of models are constructed by composing the pgfs of the geometric model with the pgfs of two specific distributions,
using the mechanism described in the Proposition 3 of that paper.

Unlike the models in Subsection 6.1, the ones considered here can only include the identity $N_I$
as an isolated point.
Consequently, the connected version of the stopping model does not make randomly stopped sum or extreme transformations
in Definitions \ref{def:rssm} and \ref{def:rsex},
into statistically stable (idempotent) extensions.
Rather than leaving the transformed statistical model unchanged, applying these transformations
twice in succession results in a ``contraction" of the transformed model.

On the other hand,
under both the stopping models in Subsection 6.1 as well as those presented here,
the four statistical model extensions introduced in Definition \ref{def:rscb}
collapse into only two distinct model extensions, that are statistically stable. These new model extensions
pair randomly stopped maxima or minima with their respective inverses, they introduce a single additional parameter
to the initial model, and they subsume the
randomly stopped extreme transformations in Definition \ref{def:rsex}.
These structural properties are a direct consequence of items 7 and 9 of Proposition \ref{prop:prpts}.

\begin{remark}
Under the models considered in Subsection 6.2 with a maximal parameter space $(0,\theta_{K_{\mathcal{N}}}]$,
for a given $N_\theta \in \mathcal{N}$ and any $n \in \mathbb{N}^+$,
an $n$-th compositional root $N_n \in \mathcal{N}$ satisfying
$h_{N_\theta}=h_{N_n}^{\circ n}$ exists if and only if
$$ p_1(N_n) = \theta^{1/n} \le \theta_{K_{\mathcal{N}}}, $$
or equivalently, if and only if
$$ n \le \frac{\log \theta}{\log \theta_{K_{\mathcal{N}}}}. $$
Consequently, the random variables in these models are
divisible under composition only up to the maximum order $\lfloor \log \theta / \log \theta_{K_{\mathcal{N}}} \rfloor$,
which depends explicitly on both $\theta$ and the threshold parameter $\theta_{K_{\mathcal{N}}}$.
This finite divisibility contrasts with the infinite divisibility under composition described in Remark 7,
which holds exclusively for the models in Subsection 6.1, which have $\theta_{K_{\mathcal{N}}}=1$.
\end{remark}

\begin{example}
For $\alpha$ in $(0,\infty)$, the family $\mathcal{N}_\alpha$ given by:
$$ \mathcal{N}_\alpha = \left\{ N_{\theta}^{\alpha}  \mid h_{N_{\theta}^\alpha}(t) =
\frac{1}{\alpha} \ln\left(1 + \frac{(e^{\alpha t} -1)(e^{\alpha} - 1)}
{(\theta^{-1} -1)(e^{\alpha} - e^{\alpha t}) + e^{\alpha} -1}\right), \quad \theta \in (0,\theta_{K_{\mathcal{N}_\alpha}}=e^{-\alpha}] \right\}, $$
is a connected stopping model closed under pgf composition that can only include $N_I$ as an isolated point.
The parameter space of this model is the maximal parameter space for the family of commuting random variables with Koenigs function:
$$ K_{\mathcal{N}_\alpha}(t) =
\frac{1}{p_1^{ZTP}} \frac{h_{N_{ZTP}}(t)}{1-h_{N_{ZTP}}(t)} =
\frac{e^{\alpha}-1}{\alpha} \frac{e^{\alpha t}-1}{e^{\alpha}-e^{\alpha t}},$$
with $p_1^{ZTP}$ and $h_{N_{ZTP}}(t)$ denoting the probability of value one and the pgf of the zero truncated poisson distribution.
\end{example}

\begin{example}
For $\alpha > 0$ and $\beta > 1$, the family $\mathcal{N}_{\alpha,\beta}$, given by:
$$ \mathcal{N}_{\alpha,\beta} = \left\{ N_{\theta}^{\alpha,\beta}  \mid h_{N_{\theta}^{\alpha,\beta}}(t) =
 \frac{1-\left(\frac{(1 - \theta^{-1} {{\rm e}^{\alpha}})
(1-t+t{{\rm e}^{-\frac{\alpha}{\beta}}})^{\beta}+ \theta^{-1} - 1}
{({{\rm e}^{\alpha}}(1-\theta^{-1}))
(1-t+t{{\rm e}^{-\frac{\alpha}{\beta}}})^{\beta}+ \theta^{-1} - {{\rm e}^{\alpha}}}\right)^{\frac{1}{\beta}}}
{1-{\rm e}^{-\frac{\alpha}{\beta}}}, \quad \theta \in (0, \theta_{K_{\mathcal{N}_{\alpha,\beta}}}=e^{-\alpha}] \right\},  $$
is a connected stopping model closed under pgf composition that can only include $N_I$ as an isolated point.
It is the maximal stopping model for the family of commuting random variables with Koenigs function:
$$ K_{\mathcal{N}_{\alpha,\beta}}(t) = \frac{1}{p_1^{ZTNB}} \frac{h_{N_{ZTNB}}(t)}{1-h_{N_{ZTNB}}(t)},    $$
where
$$ p_1^{ZTNB}= \frac{\beta}{e^{\alpha}-1} \left(1-e^{-\alpha/\beta}\right), $$
and
$$ h_{N_{ZTNB}}(t)= \frac{\left(1-(1-e^{-\frac{\alpha}{\beta}})t\right)^{-\beta}-1}{e^{\alpha}-1}, $$
denote the probability of one and the pgf of the zero truncated negative binomial distribution.
\end{example}

\begin{example}
For $\alpha > 0$ and $n \in \mathbb{N}^{+}$, the family $\mathcal{N}_{\alpha,n}$ given by:
$$  \mathcal{N}_{\alpha,n} = \left\{ N_{\theta}^{\alpha,n}  \mid h_{N_{\theta}^{\alpha,n}}(t) =
\frac{\left(
\frac{(\theta^{-1}-1)(-t+1+t{{\rm e}^{{\frac{\alpha}{n}}}})^{n} - \theta^{-1} {{\rm e}^{\alpha}} + 1}
{(\theta^{-1}-{{\rm e}^{\alpha}})(-t+1+t{{\rm e}^{{\frac{\alpha}{n}}}})^{n} + {{\rm e}^{\alpha}}(1-\theta^{-1})}
\right)^{-\frac{1}{n}}-1}{{{\rm e}^{{\frac{\alpha}{n}}}}-1},
\quad \theta \in \left(0,\theta_{K_{\mathcal{N}_{\alpha,n}}}=e^{-\alpha}\right] \right\},  $$
is a connected stopping model closed under pgf composition that can only include $N_I$ as an isolated point.
It is the maximal stopping model for the family of commuting random variables with Koenigs function:
$$ K_{\mathcal{N}_{\alpha,n}}(t) = \frac{1}{p_1^{ZTB}} \frac{h_{N_{ZTB}}(t)}{1-h_{N_{ZTB}}(t)},    $$
where
$$ p_1^{ZTB}= \frac{n}{e^{\alpha}-1}\left(e^{\frac{\alpha}{n}}-1\right), $$
and
$$ h_{N_{ZTB}}(t)=\frac{\left(1+\left(e^{\frac{\alpha}{n}}-1\right)t\right)^n-1}{e^{\alpha}-1},  $$
denote the probability of one and the pgf of the zero truncated binomial distribution.
\end{example}

\section{Construction of statistically stable (invariant) models}

Statistical models $\mathcal{Y}=\mathcal{T}(\mathcal{X})$ obtained via a statistically stable transformation $\mathcal{T}(\cdot)$
of an arbitrary model $\mathcal{X}$ are themselves stable under that transformation, meaning they
remain invariant under successive applications of $\mathcal{T}(\cdot)$. This follows directly from the definition of
stability for $\mathcal{T}(\cdot)$, which requires that $\mathcal{T}(\mathcal{T}(\mathcal{X}))=\mathcal{T}(\mathcal{X})$ for all $\mathcal{X}$.

For instance, models derived through the Marshall-Olkin extension, which is a geometrically stopped extreme extension,
are inherently stable (invariant) under subsequent applications of the same extension. This invariance holds broadly for any
model generated by a randomly stopped sum or extreme transformation, provided the underlying
stopping model is composition-closed and contains $N_I$, as is the case with the geometric.

Below, we illustrate the construction of such statistically stable models. We consider a baseline statistical model
$\mathcal{X} = \{X_{\gamma} \mid \gamma \in \mathcal{D}_\mathcal{X} \}$ characterized by its pgf $h_{X_{\gamma}}$, cdf $F_{X_{\gamma}}$, and
survival function $S_{X_{\gamma}}$, alongside two of the composition-closed stopping models $\mathcal{N}$ introduced in Section 6.1.
Using these components, we derive the explicit statistical models generated by the $\mathcal{N}$-stopped maximum, minimum and sum
of $\mathcal{X}$. Although in these examples the original model $\mathcal{X}$ will be Poisson, which is supported on the integers, the
methodology applies equally to any other discrete or continuous baseline model.

\subsection{Stable models based on Sibuya-stopped extensions}

Let the stopping model $\mathcal{N}$ be the Sibuya model from Example 1. By Definition \ref{def:rssm}, the Sibuya-stopped sum of $\mathcal{X}$ is given by
$$ \mathcal{Y}_s = sum_{\mathcal{N}}(\mathcal{X}) =
\left\{Y_{\gamma, \theta} \mid h_{Y_{\gamma, \theta}}(t) =  1 - \left(1 - h_{X_{\gamma}}(t)\right)^\theta, \quad
\gamma \in \mathcal{D}_\mathcal{X}, \theta \in (0,1] \right\}.  $$
Similarly, by Definition \ref{def:rsex}, the Sibuya-stopped maximum of $\mathcal{X}$ is defined as
$$ \mathcal{Y}_1 = max_{\mathcal{N}}(\mathcal{X}) =
\left\{Y_{\gamma, \theta} \mid F_{Y_{\gamma, \theta}}(y) = 1 - \left(1 - F_{X_{\gamma}}(y)\right)^\theta, \quad
\gamma \in \mathcal{D}_\mathcal{X}, \theta \in (0,1] \right\}, $$
and the Sibuya-stopped minimum of $\mathcal{X}$ as
$$ \mathcal{Y}_2 = min_{\mathcal{N}}(\mathcal{X}) =
\left\{Y_{\gamma, \theta} \mid S_{Y_{\gamma, \theta}}(y) = 1 - \left(1 - S_{X_{\gamma}}(y)\right)^\theta, \quad
\gamma \in \mathcal{D}_\mathcal{X}, \theta \in (0,1] \right\}. $$

As a concrete example, let the baseline model $\mathcal{X}$ be the Poisson($\gamma$) model, which has pgf $h_{X_{\gamma}}(t) = \mathrm{e}^{\gamma(t-1)}$,
cdf $F_{X_{\gamma}}(x) = \Gamma(x+1,\gamma)/x! $, and survival function $S_{X_{\gamma}}(x) = 1-\Gamma(x+1,\gamma)/x!$, where $x \in \mathbb{N}$ and
$\Gamma(\cdot,\cdot)$ denotes the upper incomplete Gamma function.
The resulting Sibuya-stopped sum Poisson is the model
$$ \mathcal{Y}_s = \left\{Y_{\gamma, \theta} \mid h_{Y_{\gamma, \theta}}(t) =  1 - \left(1 - {\mathrm{e}}^{\gamma(t-1)}\right)^\theta, \quad
\gamma \in \mathcal{D}_\mathcal{X}, \theta \in (0,1] \right\},  $$
the Sibuya-stopped maximum Poisson is
$$ \mathcal{Y}_1 = \left\{Y_{\gamma, \theta} \mid F_{Y_{\gamma, \theta}}(y) = 1 - \left(1 - \frac{\Gamma(y+1,\gamma)}{y!}\right)^\theta, \quad
\gamma \in \mathcal{D}_\mathcal{X}, \theta \in (0,1] \right\}, $$
and the Sibuya-stopped minimum Poisson is
$$ \mathcal{Y}_2 = \left\{Y_{\gamma, \theta} \mid S_{Y_{\gamma, \theta}}(y) = 1 - \left(\frac{\Gamma(y+1,\gamma)}{y!}\right)^\theta, \quad
\gamma \in \mathcal{D}_\mathcal{X}, \theta \in (0,1] \right\}. $$
Each of these three models is invariant under repeated applications of its respective Sibuya-stopped extension.
Specifically, applying the Sibuya-stopped sum, maximum or minimum extension twice in succession to the Poisson
baseline model with parameters $\theta_1, \theta_2 \in (0,1]$ simply reproduces the same models with updated parameters $(\gamma, \theta_1 \theta_2)$.

\subsection{Stable models based on extensions using Example 4 with $\alpha=1/2$}

Let $\mathcal{N}$ be the stopping model from Example 4 with parameter $\alpha=1/2$, defined by
$$ \mathcal{N} = \left\{ N_\theta \mid h_{N_\theta}(t) = \frac{2}{1+\sqrt{1+\frac{4(1-t)}{t^2\theta^2}}}, \quad \theta \in (0,1]  \right\}. $$
Applying this stopping model to the Poisson$(\gamma)$ baseline yields the following three invariant families.
The $\mathcal{N}$-stopped sum Poisson is
$$ \mathcal{Y}_s = \left\{Y_{\gamma, \theta} \mid h_{Y_{\gamma, \theta}}(t) =
\frac{2}{1+\sqrt{1+\frac{4(1-{\mathrm{e}}^{\gamma(t-1)})}{{\mathrm{e}}^{2\gamma(t-1)}\theta^2}}}, \quad
\gamma \in \mathcal{D}_\mathcal{X}, \theta \in (0,1] \right\},  $$
the $\mathcal{N}$-stopped maximum Poisson is
$$ \mathcal{Y}_1 = \left\{Y_{\gamma, \theta} \mid F_{Y_{\gamma, \theta}}(y) =
\frac{2}{1+\sqrt{1+    \frac{4 y! \left(y! - \Gamma(y+1,\gamma)\right)}{\Gamma(y+1,\gamma)^2\theta^2}}}, \quad
\gamma \in \mathcal{D}_\mathcal{X}, \theta \in (0,1] \right\}, $$
and the corresponding $\mathcal{N}$-stopped minimum Poisson is
$$ \mathcal{Y}_2 = \left\{Y_{\gamma, \theta} \mid S_{Y_{\gamma, \theta}}(y) =
\frac{2}{1 + \sqrt{1 + \frac{4 y! \Gamma(y+1,\gamma)}{\left(y!-\Gamma(y+1,\gamma)\right)^2\theta^2}}}, \quad
\gamma \in \mathcal{D}_\mathcal{X}, \theta \in (0,1] \right\}. $$
These three models again remain strictly invariant under repeated applications of their corresponding randomly stopped transformations.
Successive applications with parameters $\theta_1, \theta_2 \in (0,1]$ regenerate the same functional forms reparametrized to $(\gamma, \theta_1 \theta_2)$.

Note that extending the parameter space of $\theta$ from $(0,1]$ to $(0,\infty)$ makes the statistical models labeled $\mathcal{Y}_1$ coincide with
those obtained via the first two transformations in Definition \ref{def:rscb}, and similarly makes
the $\mathcal{Y}_2$ models coincide with those obtained via the last two transformations in that definition.
This follows as an immediate consequence of items 7 and 9 of Proposition \ref{prop:prpts}.

\section{Discussion}

The statistical stability of randomly stopped sums and extreme model extensions requires the underlying stopping model to be closed under
pgf composition. We establish that imposing the additional conditions of finite dimensionality, continuous differentiability, positivity,
connectedness, and a non-empty interior restricts the stopping model to a family of commuting random variables whose support includes $1$.
Specifically, finite-dimensionality links composition-closure directly to commutativity, which in turn ties
commutativity to the idempotency of randomly stopped sums and extreme transformations.

These commuting families are uniquely characterized by their Koenigs function, $K_{\mathcal{N}}$.
Any function analytic on $(0,1)$ with non-negative Taylor coefficients,
satisfying $K_{\mathcal{N}}(0)=0$, $K_{\mathcal{N}}^{\prime}(0)=1$, and $\lim_{t \to 1^-} K_{\mathcal{N}}(t) = \infty$,
constitutes a valid normalized Koenigs function associated with a distinct commuting family.
Depending on the analytic properties of $K_{\mathcal{N}}$, the maximal parameter space $\mathcal{D}_{\mathcal{N}}$ of
stopping models generated by this family falls into one of four categories:
\begin{description}
\item (i) $\mathcal{D}_{\mathcal{N}}=(0,1]$, which includes the identity $N_I$ as part of the connected model;
\item (ii) $\mathcal{D}_{\mathcal{N}}=(0,\theta_{K_{\mathcal{N}}}]$ for some $\theta_{K_{\mathcal{N}}} < 1$,
which admits $N_I$ only as an isolated point;
\item (iii) $\mathcal{D}_{\mathcal{N}} \subset \mathbb{N}^{+}$, yielding a discrete parameter space; or
\item (iv) $\mathcal{D}_{\mathcal{N}}=\{1\}$, the degenerate case in which $\mathcal{N}$ consists solely of $N_I$.
\end{description}
While Theorem \ref{thm:nnn} characterizes the Koenigs functions corresponding to category (i),
characterizing the Koenigs functions associated with the other three categories remains an open problem.
Ultimately, a general criterion is needed to determine the maximal admissible parameter space directly
from $K_\mathcal{N}(t)$ and to distinguish the commuting families of random variables that admit a connected parameter space
from those confined to discrete or degenerate domains.

The statistical models generated via statistically stable transformations are inherently
invariant under repeated applications of those transformations.
Consequently, employing randomly stopped sum and extreme transformations with composition-closed stopping models
opens the door to constructing rich new families of invariant statistical models,
substantially generalizing the classical Marshall-Olkin framework.

We view the assumptions required to prove our main results (finite-dimensionality, connectedness, non-empty interior, and the continuous
differentiability of probability mass functions) as a positive feature rather than a limitation. These are precisely the properties one expects
from standard parametric models in applied statistics. We adopted these assumptions primarily because we wanted our models to possess
these desirable statistical features in practice, rather than strictly out of mathematical necessity. It is encouraging that the
assumptions required for our theoretical results align perfectly with standard applied modeling requirements.

Finally, we emphasize that our analysis restricts the stopping models to those supported strictly on the
positive integers. As detailed in Remark 2,
this restriction on the stopping model is mathematically necessary for randomly stopped extreme transformations.
While randomly stopped sum transformations can accommodate stopping models that include zero in their support,
such zero-inclusive stopping models fall outside the scope of the present paper.
As a consequence, the statistical stability of randomly stopped sum extensions that use stopping models with zero in their support
has been excluded from consideration.

\section*{Appendix 1: Proof of propositions in Section 3}

\begin{proof}[Proof of Proposition \ref{prop:zn1}]
Closure of $\mathcal{N}_{N_0}$ under composition follows directly from the associativity of
functional composition and the fact that if $h_{N_1}$ and $h_{N_2}$ commute with $h_{N_0}$,
then their composition $h_{N_1} \circ h_{N_2}$ does as well.
The inclusion of the iterates $\{N_0^{\circ m}\}$ is trivial as every function commutes with its own iterates.
\end{proof}

\begin{proof}[Proof of Proposition \ref{prop:znz}]
Let $h_{N_0}(t) = \sum_{j=m_0}^{\infty} p_j t^j$ and $h_{N}(t) = \sum_{j=m}^{\infty} a_j t^j$
be the pgfs of $N_0$ and $N$, with $m_0, m \ge 1$ and $p_{m_0}, a_m > 0$.
We focus on the case where one variable has no support at 1; without loss of generality, assume $p_1(N_0) = 0$, so that $m_0 > 1$.

By comparing the coefficients of the lowest-order terms in the identity $h_N \circ h_{N_0} = h_{N_0} \circ h_N$, we obtain $a_1 p_{m_0} = p_{m_0} a_1^{m_0}$. For $m_0 > 1$, this implies $a_1 \in \{0, 1\}$. Excluding the identity case $a_1=1$, yields $a_1= p_1(N) = 0$.
The converse follows by symmetry.
\end{proof}

\begin{proof}[Proof of Proposition \ref{prop:zn2}]
Let $N_0$ be a random variable with pgf $h_{N_0} = \sum_{i=m_0}^{\infty} p_i t^i$, where $m_0 > 1$ and $p_{m_0} > 0$.
Consider two random variables, $N_1, N_2 \in \mathcal{N}_{N_0}$,
with pgfs $h_{N_1}=\sum_{i=m}^{\infty}a_i t^i$ and $h_{N_2}=\sum_{i=m}^{\infty}\hat{a}_i t^i$, both
sharing the same leading order $m$, with
$a_m, \hat{a}_m \ne 0$. We show by induction that $a_k=\hat{a}_k$ for all $k \ge m$.

Since $N_1, N_2 \in \mathcal{N}_{N_0}$, the commutation relation $h_{N_0}\circ h_{N_j}(t) = h_{N_j}\circ h_{N_0}(t)$ implies:
$$  \sum_{i=m_0}^{\infty}p_i\left(\sum_{k=m}^{\infty}a_k t^k\right)^i = \sum_{k=m}^{\infty}a_{k}\left(\sum_{i=m_0}^{\infty}p_{i}t^{i}\right)^{k}. $$
In the base case $k=m$, equating the coefficients of the lowest-order, $t^{m_0 m}$, on both sides,
\begin{enumerate}
\item For $h_{N_0}\circ h_{N_1}$, the leading term arises from $i=m_0$, yielding $p_{m_0} a_m^{m_0}$.
\item For $h_{N_1}\circ h_{N_0}$, the leading term arises from $k=m$, yielding $a_m p_{m_0}^m$.
\end{enumerate}
Equating these gives $p_{m_0} a_m (a_m^{m_0-1} - p_{m_0}^{m-1}) = 0$.
Since $p_{m_0}, a_m \neq 0$ and $m, m_0 > 1$, the leading coefficient is
$a_m = p_{m_0}^{\frac{m-1}{m_0-1}}$. As this value is uniquely determined by $N_0$ and the index $m$,
it also holds for $N_2$ and $\hat{a}_m = a_m$.

In the inductive step, assume $\hat{a}_k = a_k$ for all $m \le k < m+n$, for some $n \ge 1$.
Let $[f(t)]_k$ denote the coefficient of $t^k$ in the expansion of $f(t)$.
We examine the coefficient of order $t^{m_0 m + n}$ in the commutation equation.

In $h_{N_0}\circ h_{N_1}$, the term $a_{m+n}$ first appears in $p_{m_0}(h_{N_1}(t))^{m_0}$.
Specifically, its contribution to the coefficient of $t^{m_0 m +n}$
arises from the multinomial expansion where
$a_m$ is selected $m_0-1$ times and $a_{m+n}$ is selected once, resulting in $p_{m_0} m_0 a_m^{m_0-1} a_{m+n}$. All other terms of
this order depend only on the coefficients $\{a_k\}_{k < m+n}$. Hence,
$$ [h_{N_0} \circ h_{N_1}]_{m_0 m + n} = p_{m_0} m_0 a_m^{m_0-1} a_{m+n} + Q_{m,n}(p, a_m, \dots, a_{m+n-1}), $$
where $Q_{m,n}$ is a polynomial involving only lower-order previously determined coefficients.

In $h_{N_1}\circ h_{N_0}$, the coefficient $[h_{N_1} \circ h_{N_0}]_{m_0 m + n}$ depends only on
coefficients up to $a_{m+n-1}$. To see this, note that the term involving $a_{m+n}$
is $a_{m+n}(h_{N_0}(t))^{m+n}$, and its lowest power is $t^{m_0(m+n)}$.
Since $m_0>1$ and $n>1$, it follows that $m_0 m + m_0 n > m_0 m + n$. Thus,
$a_{m+n}$ does not contribute to the coefficient of order $m_0 m + n$ on the right-hand side.

Equating the coefficients of order $m_0 m + n$ from both sides yields:
$$ p_{m_0} m_0 a_m^{m_0-1} a_{m+n} = R_{m,n}(p, a_m, \dots, a_{m+n-1}), $$
where $R_{m,n}$ is a combination of polynomials from both sides of the expansion.
Solving for $a_{m+n}$,
$$ a_{m+n} = \frac{R_{m,n}(p,a_m,\ldots,a_{m+n-1})}{p_{m_0}m_0 a_m^{m_0-1}}. $$
Since $p_{m_0}, m_0, a_m \neq 0$, the coefficient $a_{m+n}$ is uniquely determined by $N_0$ and the set $\{a_k\}_{k < m+n}$.
By the inductive hypothesis, $\hat{a}_{m+n}=a_{m+n}$.
By induction, $h_{N_1} = h_{N_2}$, and
the model $\mathcal{N}_{N_0}$ is uniquely parametrized by the integer $m \in \mathbb{N}^+$.
\end{proof}

\begin{proof}[Proof of Proposition \ref{prop:gincrb}]
Constant functions $g(t) \equiv 0$ and $g(t) \equiv 1$ trivially
commute with $h_{N_0}$. For non-constant analytic $g$, we show $g'(t) > 0$ on $(0,1)$.

Suppose $g'(t_0) = 0$ for some $t_0 \in (0,1)$.
Differentiating $g \circ h_{N_0} = h_{N_0} \circ g$ yields:
$$ g'(h_{N_0}(t)) h'_{N_0}(t) = h'_{N_0}(g(t)) g'(t). $$
Since $h'_{N_0} > 0$ on $(0,1)$, $g'(t_0) = 0$ implies $g'(h_{N_0}(t_0)) = 0$. Let $t_n = h_{N_0}^{\circ n}(t_0)$.
By induction, $g'(t_n) = 0$ for all $n \in \mathbb{N}$. Because $h_{N_0}$ is the pgf of a positive integer-valued variable,
$t_n \to 0$ as $n \to \infty$. The analyticity of $g$ implies its zeros cannot accumulate at $0$ unless $g'$ is identically zero,
which contradicts our non-constant assumption. Thus, $g'(t) \neq 0$ on $(0,1)$.
Since $g$ must map $[0,1]$ to its domain and $h_{N_0}$ is increasing, $g$ must be strictly increasing.

For the boundaries, commutativity at the fixed points of $h_{N_0}$ requires $h_{N_0}(g(1)) = g(h_{N_0}(1)) = g(1)$
and $h_{N_0}(g(0)) = g(h_{N_0}(0)) = g(0)$. Since $0$ and $1$ are the unique fixed points of $h_{N_0}$ on $[0,1]$,
and $g$ is strictly increasing, it follows that $g(0)=0$ and $g(1)=1$. Thus, $g$ is a strictly increasing bijection of $[0,1]$.
\end{proof}

\begin{proof}[Proof of Proposition \ref{prop:com}]
By Theorem \ref{thm:par}, every $g \in \mathcal{G}_{N_0}$ is uniquely determined by $\theta = g'(0)$.
Since $\mathcal{G}_{N_0}$ is closed under composition (Proposition \ref{prop:gcfc}),
$g_{\theta_1} \circ g_{\theta_2} \in \mathcal{G}_{N_0}$.
By the chain rule, $(g_{\theta_1} \circ g_{\theta_2})'(0) = g_{\theta_1}'(0)g_{\theta_2}'(0) = \theta_1 \theta_2$.
The unique parametrization implies $g_{\theta_1} \circ g_{\theta_2} = g_{\theta_1 \theta_2}$.

For $0 < p_1(N_0) < 1$, $g_{\theta_1}$ is a strictly increasing bijection (Proposition \ref{prop:gincrb}).
Its inverse $g_{\theta_1}^{-1}$ satisfies
$g_{\theta_1}^{-1} \circ h_{N_0} = g_{\theta_1}^{-1} \circ (h_{N_0} \circ g_{\theta_1}) \circ g_{\theta_1}^{-1} = g_{\theta_1}^{-1}
\circ (g_{\theta_1} \circ h_{N_0}) \circ g_{\theta_1}^{-1} = h_{N_0} \circ g_{\theta_1}^{-1}$,
thus $g_{\theta_1}^{-1} \in \mathcal{G}_{N_0}$.
Since $(g_{\theta_1}^{-1})'(0) = 1/g_{\theta_1}'(0) = 1/\theta_1$,
we have $g_{\theta_1}^{-1} = g_{1/\theta_1}$.
The third statement follows immediately from the first two.
\end{proof}

\begin{proof}[Proof of Proposition \ref{prop:finit}]
We proceed by contradiction.  Assume $\mathcal{N}_{N_0}$ has a non-empty interior.
By Theorem \ref{thm:intp}, there exists $\theta_{N_0}>0$ such that $(0,\theta_{N_0}]\subset\mathcal{D}_{N_0}$.
Let $\theta_0 = h_{N_0}^{\prime}(0)$ and choose $n \ge 1$ such that $\theta_1 = \theta_0^n \in (0,\theta_{N_0}]$.
Define $N_1 = N_0^{\circ n}$, so that $h_{N_1}^{\prime}(0) = \theta_1$.

Since $N_0$ has finite support, $N_1$ also has finite support. Its pgf is a polynomial of degree $r$,
$$ h_{N_1}=g_{\theta_1}=\sum_{n=1}^r a_n(\theta_1) t^n, \quad a_r(\theta_1)\neq0.  $$
Because $N_1$ is an interior point, any infinitesimal variation:
$$  \frac{\partial}{\partial\theta}g_{\theta}(t)_{|\theta=\theta_1} = \sum_{n=1}^{r_2} d_n(\theta_1)t^n,  $$
must be a polynomial of degree $r_2 \le r$; If there existed a non-zero coefficient $d_j(\theta_1)$ for
$j>r$, the coefficients of $g_{\theta_1 \pm \epsilon}(t)$ would eventually become negative for a sufficiently small $\epsilon$,
contradicting that $h_{N_1}$ is in the interior of $\mathcal{H}_{N_0}$.

Differentiating $g_{\theta} \circ h_{N_1} =h_{N_1}\circ g_{\theta}$ w.r.t. $\theta$ and evaluating at $\theta=\theta_1$, we have:
$$ \frac{\partial}{\partial\theta}\left(g_{\theta}\circ h_{N_1}\right)_{|\theta=\theta_1} = \frac{\partial}{\partial\theta}\left(h_{N_1}\circ g_{\theta}\right)_{|\theta=\theta_1}.$$
The left hand side is:
$$ \frac{\partial}{\partial\theta}\left(g_{\theta}\circ h_{N_1}\right)_{|\theta=\theta_1} = \left(\sum_{n=1}^{r_2} d_n\left(\theta_1\right) t^n\right) \circ \left(\sum_{n_2=1}^r a_{n_2}\left(\theta_1\right) t^{n_2}\right), $$
where the leading term is $d_{r_2}(\theta_1) a_r^{r_2}(\theta_1) t^{r_2 r}$.
By the chain rule, the right hand side is:
$$\frac{\partial}{\partial\theta}\left(h_{N_1} \circ g_{\theta}\right)_{|\theta=\theta_1} = \left(\left(\sum_{n=1}^r n a_n\left(\theta_1\right) t^{n-1}\right) \circ \left(\sum_{n_2=1}^r a_{n_2}\left(\theta_1\right) t^{n_2}\right)\right) \left(\sum_{n_3=1}^{r_2} d_{n_3}\left(\theta_1\right) t^{n_3}\right), $$
where the leading term is $r a_r^r(\theta_1) d_{r_2}(\theta_1) t^{r(r-1) + r_2}$.
Equating degrees, $r_2 r = (r-1) r + r_2$, we find $r_2=r$ because $r>1$.
Equating coefficients, $d_r(\theta_1) a_r^r(\theta_1) = r a_r^r(\theta_1) d_{r_2}(\theta_1)$, we find
$(r-1) a_r(\theta_1)^r d_r(\theta_1)=0$, but that is impossible, because $r>1$, $a_r(\theta_1) \neq 0$ and $d_r(\theta_1) \neq 0$.
Hence we reach a contradiction, which implies that the interior of $\mathcal{N}_{N_0}$ must be empty.
\end{proof}

\section*{Appendix 2: Proofs of Propositions in Section 4}

\begin{proof}[Proof of Proposition \ref{prop:tpr1}]
We aim to show that if $\mathcal{N} \neq \mathcal{N}_1$, then the model is disconnected.
First, observe that $\mathcal{N} \neq \mathcal{N}_1$ is equivalent to the complement
$\mathcal{N}_o^c$ (relative to $\mathcal{N}$) being non-empty, and that $\mathcal{N}_o$ is non-empty by definition.

Recall the partition
$\mathcal{N}_o^c = \{ N\in \mathcal{N} \mid p_o(N)=0 \} = \overline{\mathcal{N}}_o^c \cup (\overline{\mathcal{N}_o}\cap\mathcal{N}_o^c)$.
We partition the model as
$\mathcal{N}=\mathcal{N}_o \cup \overline{\mathcal{N}}_o^c \cup (\overline{\mathcal{N}_o}\cap\mathcal{N}_o^c)$.
If we establish that $\overline{\mathcal{N}_o}\cap\mathcal{N}_o^c = \emptyset$ whenever
$\mathcal{N} \neq \mathcal{N}_1$, then $\mathcal{N}$ simplifies to the union of two disjoint open sets,
$\mathcal{N}=\mathcal{N}_o \cup \overline{\mathcal{N}}_o^c$, proving that $\mathcal{N}$ is disconnected.

We argue by contradiction.
Suppose that there exists an element $N_{0} \in \overline{\mathcal{N}}_o \cap \mathcal{N}_o^c$.
Since $N_{0} \in \mathcal{N}_o^c$, it holds that $N_0 \in \mathcal{N}_s \subset \mathcal{N}_o^C \subset \mathcal{N}$
for some $s > o \ge 1$, and its pgf is $h_{N_{0}}(t)=\sum_{j=s}^\infty p_j(N_0) t^j$ with $p_s(N_0) >0$.
Since $N_{0} \in \overline{\mathcal{N}_o}$,
there exists a sequence $\{N_i \}_{i \in \mathbb{N}} \subset \mathcal{N}_o$ such that
$N_{i} \rightarrow N_0$, which implies $p_o(N_i) \rightarrow p_o(N_0) = 0$.

Since $\mathcal{N}_o$ is an open subset of $\mathcal{N}$, it has a non-empty interior and there exists
an open ball $\mathcal{B}_c \subset \mathcal{N}_o$ centered at some $N_c$.
Moreover, the set of compositions $\mathcal{B}_c \circ N_0$ is contained within $\mathcal{N}_{s o}$,
since $N_c \in \mathcal{N}_o$ and $N_0 \in \mathcal{N}_s$.
Because composition with $N_0$ acts as a homeomorphism on the parameter space, $\mathcal{B}_c \circ N_0$ is an open subset
of $\mathcal{N}$, consisting of interior points of $\mathcal{N}_{so}$.

Now, consider the sequence of compositions $\{N_c \circ N_i\}_{i \in \mathbb{N}}$.
By the continuity of pgf composition, since $N_c \in \mathcal{N}_o$ and $N_i \in \mathcal{N}_o$,
the entire sequence $\{N_c \circ N_i\}$ is contained in $\mathcal{N}_{o^2}$.
Furthermore, as $i \to \infty$, $N_c \circ N_i$ converges to $N_c \circ N_0$ which is an interior point of $\mathcal{N}_{so}$.

By the definition of an interior point, there must exist an index $i_0$ such that for all $i > i_0$, the element $N_c \circ N_i$
lies within an open neighborhood of $N_c \circ N_0$ contained in $\mathcal{N}_{so}$. However, this requires
$N_c \circ N_i \in \mathcal{N}_{o^2} \cap \mathcal{N}_{so}$. Since $s > o \ge 1$, it follows that $so > o^2$,
and thus $\mathcal{N}_{o^2} \cap \mathcal{N}_{so} = \emptyset$. This contradiction shows that no such $N_0$ can exist.

Therefore, if $\mathcal{N} \neq \mathcal{N}_1$ then
$\mathcal{N}=\mathcal{N}_o \cup \overline{\mathcal{N}}_o^c$, which
is the union of two disjoint open sets, which means that $\mathcal{N}$ is not connected.
We conclude that $\mathcal{N}$ can be connected only if $\mathcal{N} = \mathcal{N}_1$
\end{proof}

\begin{proof}[Proof of Proposition \ref{prop:ff1}]
%
Since $p_1(N)\in (0,1)$ is not constant, because $p_1(N^{\circ n}) = p_1(N)^n \ne p_1(N)$, there exists
$N_0 \in \mathcal{N}_1$ such that $\nabla p_1(\delta) \ne 0$ at the local parameter value $\delta_0=\delta(N_0)$.
Choosing an index $j$ such that $\frac{\partial p_1}{\partial\delta_j} \ne 0$,
we define the sub-vector $\delta^*=(\delta_1, \ldots, \delta_{j-1},\delta_{j+1}, \ldots,\delta_m)$,
where $m$ is the dimension of the stopping model, $\mathcal{N}$.
The Jacobian of the mapping $\delta \rightarrow \left(p_1(\delta),\delta^*(\delta)\right)$
is non-singular, satisfying
$\det\left(J\right) = \frac{\partial p_1}{\partial\delta_j} \ne 0$.

By the Inverse Function Theorem,
there exists an open neighborhood $\mathcal{U}$ of $N_0$ admitting a valid local parametrization given by $(\theta, \delta^*)$, where $\theta=p_1(N)$.
Analogously, there exists an open neighborhood
$\mathcal{V}$ of $N_0\circ N_0$ with valid local $(\theta,\delta^{*})$-parametrization, and
$\mathcal{U}$ can be such that it is sufficiently small to satisfy $(N_0 \circ \mathcal{U})\cup (\mathcal{U} \circ N_0) \subset \mathcal{V}$.
For the set $N_0 \circ \mathcal{U}$ we adopt the $(\theta_0 \theta, \delta^{*}_{N_0 \circ \mathcal{U}})$ parametrization,
while for $\mathcal{U}\circ N_0$, we use the $(\theta_0 \theta, \delta^{*}_{\mathcal{U} \circ N_0})$ parametrization.

Using these coordinates, we define the continuously differentiable map $D: \mathcal{U} \longrightarrow \mathbb{R}^{m-1}$
by $N_{(\theta, \delta^*_{\mathcal{U}})} \mapsto \delta^*_{N_0 \circ \mathcal{U}} - \delta^*_{\mathcal{U} \circ N_0}$,
whose image corresponds to the last $m-1$ components of $(\theta_0 \theta, \delta^*_{N_0 \circ \mathcal{U}}) - (\theta_0 \theta, \delta^*_{\mathcal{U} \circ N_0})$.
Since $D$ is differentiable and $D(\theta_0, \delta^*_0)=0$,
the Implicit Function Theorem yields a curve
$N_{(r \theta_0, \delta^*_{\mathcal{U}}(r))}$ in $\mathcal{U}$ defined for
$r \in (1-\epsilon,1+\epsilon)$, satisfying $\delta^{*}_{\mathcal{U}}(1) = \delta^*_0$
and $D(r \theta_0, \delta^{*}_{\mathcal{U}}(r)) \equiv 0$.
Because the identity $D(\theta_1,\delta^*_1)=0$ is equivalent to $N_0$ and $N_{(\theta_1,\delta^*_1)}$ commuting, we have that:
$$ \mathcal{R}_0 = \{ \mathcal{R}_0(r)=N_{(r \theta_0, \delta^*_{\mathcal{U}}(r))}  \mid r \in (1-\epsilon, 1+\epsilon) \} \subset \mathcal{N},  $$
with $\mathcal{R}_0(1)=N_0$, and we also have that $\mathcal{R} \subset \mathcal{N}_{N_0}$. Consequently, $\mathcal{R}_0 \subset \mathcal{N} \cap \mathcal{N}_{N_0}$.
\end{proof}

The proof of Proposition $\ref{prop:ff2}$ relies on the following lemma.
\begin{lemma}
Let $f$ be a real-valued analytic function on $[0, \infty)$ satisfying $f(0)=0$, $f^{\prime}(0)>0$, and $\lim_{t \to \infty} f(t) = \infty$.
For each $n \in \mathbb{N}^{+}$, consider the family of functions
$$ \mathcal{G}_{\theta,f,n}^{*} =
\left\{ g_{\theta_1,\ldots,\theta_n}^{*}(t) = (\theta_1 f) \circ (\theta_2 f) \circ \dots \circ (\theta_{n-1} f) \circ (\theta_n t) \mid
(\theta_1, \ldots, \theta_n) \in (\mathbb{R}^{+})^n  \right\}, $$
and let $d(n)$ denote the dimension of its parameter space (i.e. the number of independent effective parameters).
Then, unless $f(t)=Kt$ for some constant $K>0$, $d(n)$ grows unboundedly as $n \to \infty$.
\label{lem:dmn}
\end{lemma}

\begin{proof}[Proof of Lemma \ref{lem:dmn}]
For the composition:
$$ g_{\theta_1,\ldots,\theta_n}^{*}(t) = (\theta_1 f) \circ (\theta_2 f) \circ \dots \circ (\theta_{n-1} f) \circ (\theta_n t) $$
to collapse to a finite number of effective parameters as $n \to \infty$,
the family of transformations $x \mapsto \theta f(x)$ must generate a finite-dimensional continuous Lie transformation group on $\mathbb{R}$.
According to Lie's classification theorem for vector fields on $\mathbb{R}$, the maximum possible dimension of any such transformation group is three.
Therefore, the most general function that permits this parameter reduction to a finite-dimensional space for any $n \ge 4$ takes the form:
$$ f(t) = \left(\frac{A t^k + B}{Ct^k + D}\right)^{1/k},   $$
with $k>0$ and $A D - B C \neq 0$ (see, e.g., Aczel, 1966; Kuczma et al., 1990).

Evaluating this function at the origin yields:
$$ f(0) = \lim_{t \to 0^{+}} \left(\frac{At^{k}+B}{Ct^{k}+D}\right)^{1/k} = \left(\frac{B}{D}\right)^{1/k}. $$
The condition $f(0)=0$ thus forces $B=0$, which reduces the function to:
$$ f(t) = \left(\frac{A t^k}{Ct^k + D}\right)^{1/k}.  $$
Examining the asymptotic behavior, we observe that for $C \neq 0$:
$$ \lim_{t \to \infty} f(t) = \lim_{t \to \infty} \left(\frac{A t^k}{Ct^k + D}\right)^{1/k} = \left(\frac{A}{C}\right)^{1/k}. $$
For the requirement $\lim_{t \to \infty} f(t) = \infty$ to hold, we must have $C=0$. This constraint further simplifies the function to
a linear form:
$$ f(t) = \left(\frac{A}{D}\right)^{1/k} t = K t,   $$
with $K=(A/D)^{1/k}$. Finally the condition $f^{\prime}(0)>0$ dictates that $K>0$. Therefore parameter reduction to a finite-dimensional space
when $n \to \infty$ is only feasible if $f(t) = K t$ for some constant $K > 0$.
\end{proof}

\begin{proof}[Proof of Proposition \ref{prop:ff2}]
Recall that for $g_{\theta} \in \mathcal{G}_{N_0}$, we have the representation:
$$ g_{\theta}(t) = K_{N_{0}}^{-1}(\theta  K_{N_{0}}(t)), $$
where $\theta=g_\theta^\prime(0) \in (0,\infty)$, and $K_{N_{0}}$ is the Koenigs function associated with $N_0$.

If $h_{N} \in \mathcal{H}_\mathcal{N}$ does not commute with $h_{N_0}$,
it cannot commute with the elements of $\mathcal{G}_{N_0}$. This implies that $K_{N_0}\left(h_{N}(K_{N_0}^{-1}(t))\right) \neq h_{N}^{\prime}(0) t,$
and hence the function $f = K_{N_0} \circ h_N \circ K_{N_0}^{-1}$ is not a simple linear map of the form $f(t) = Kt$ with $K>0$.

We can express the $n$-fold composition as:
$$ g_{\theta_1, \ldots, \theta_n} = g_{\theta_1} \circ h_{N} \circ g_{\theta_2} \circ h_{N} \circ \dots \circ h_{N} \circ g_{\theta_n} =  $$
$$ K_{N_{0}}^{-1} \circ \left(\theta_1 f\right) \circ \left(\theta_2 f\right) \circ \dots \circ \left(\theta_{n-1} f\right) \circ \left(\theta_n\,t\right) \circ K_{N_{0}},  $$
where $f = K_{N_0} \circ h_{N} \circ K^{-1}_{N_0}$. Observe that $f$ is analytic on $[0,\infty)$, is not of the form $f(t) = K t$ with $K>0$, and
satisfies the remaining hypotheses of Lemma \ref{lem:dmn} because,
\begin{itemize}
\item $f(0) = K_{N_0}(h_N(K^{-1}_{N_0}(0))) = K_{N_0}(h_N(0)) = K_{N_0}(0) = 0$;
\item $f^{\prime}(0) = K^{\prime}_{N_0}(h_N(K^{-1}_{N_0}(0))) \cdot h^{\prime}_N(K^{-1}_{N_0}(0)) \cdot (K^{-1}_{N_0})^{\prime}(0) =
K^{\prime}_{N_0}(0) \cdot h^{\prime}_N(0) \cdot (K^{-1}_{N_0})^{\prime}(0) = \theta_N$; and
\item $\lim_{t \to \infty} f(t) = \infty$, because $\lim_{t \to \infty} K^{-1}_{N_0}(t)=1$,
$\lim_{s \to 1} h_N(s) = 1$, and $\lim_{s \to 1}K_{N_0}(s) = \infty$.
\end{itemize}
To analyze the dimensionality of the set of these functions, $\mathcal{G}_{N_0, N, n}$, we define the transformation
$$ g_{\theta_1, \ldots, \theta_n}^{*} = K_{N_0} \circ g_{\theta_1, \ldots, \theta_n} \circ K_{N_{0}}^{-1} = \left(\theta_1 f\right)\circ\left(\theta_2 f\right)\circ \dots \circ\left(\theta_{n-1} f\right)\circ\left(\theta_n t\right). $$
The set of these transformed functions, $ \mathcal{G}^{*}_{N_0,N,n} =
\{ g_{\theta_1,\ldots,\theta_n}^{*} \mid (\theta_1, \ldots, \theta_n) \in (\mathbb{R}^{+})^n \},$
is the image of $\mathcal{G}_{N_0, N, n}$ under a fixed pre- and post-composition with $K_{N_0}$ and $K_{N_0}^{-1}$,
which are bijections on the space of analytic functions. This implies that the dimension of $\mathcal{G}_{N_0, N, n}$
is identical to the dimension of $\mathcal{G}^*_{N_0, N, n}$.
By Lemma \ref{lem:dmn}, the dimension of the set $\mathcal{G}^*_{N_0, N, n}$ grows unboundedly when $n \to \infty$, and so does
the dimension of the original set $\mathcal{G}_{N_0, N, n}$.
\end{proof}

\section*{Appendix 3: Proof of the Proposition in Section 5}

\begin{proof}[Proof of Proposition \ref{prop:vphi1}]
The function $K_{\mathcal{N}}(t)$ is analytic on $[0,1)$, and satisfies $K_{\mathcal{N}}(0)=0$, $K_{\mathcal{N}}^\prime(0)=1$,
and $K_{\mathcal{N}}^{(n)}(0) \ge 0$ for all $n$.
We may therefore represent $K_{\mathcal{N}}(t)$ by:
$$ K_{\mathcal{N}}(t) = t + \sum_{i=2}^{\infty}a_it^i.  $$
Substituting this expansion into the definition of $\varphi_{\mathcal{N}}(t)$,
$$ \varphi_{\mathcal{N}}(t) =
t\frac{\sum_{i=2}^{\infty}(i-1)\,a_it^{i-2}}{1+\sum_{i=2}^{\infty}i\,a_it^{i-1}}. $$
This expression confirms that $\varphi_{\mathcal{N}}(t)$ is analytic on $[0,1)$ and
$\varphi_{\mathcal{N}}(0)=0$.

Since $\lim_{t\rightarrow 1^{-}}K_{\mathcal{N}}(t)=\infty$, it follows that
$\lim_{t\rightarrow 1^{-}} \log K_{\mathcal{N}}(t)=\infty$, and hence that
$$ \lim_{t\rightarrow 1^{-}} (\log K_{\mathcal{N}}(t))^\prime =
\lim_{t\rightarrow 1^{-}} \frac{K_{\mathcal{N}}^\prime(t)}{K_{\mathcal{N}}(t)} = \infty,  $$
which implies
$$ \lim_{t\rightarrow 1^{-}} \frac{K_{\mathcal{N}}(t)}{K_{\mathcal{N}}^\prime(t)} = 0.  $$
Applying this to the definition of $\varphi_{\mathcal{N}}(t)$ yields,
$$  \lim_{t\rightarrow 1^{-}}  \varphi_{\mathcal{N}}(t) = \lim_{t\rightarrow 1^{-}} 1 - \frac{K_{\mathcal{N}}(t)}{t K_{\mathcal{N}}^{\prime}(t)} = 1.  $$
To determine the derivative $\varphi_{\mathcal{N}}^\prime(1)$, we differentiate $\varphi_{\mathcal{N}}(t)$,
$$ \varphi_{\mathcal{N}}^\prime(t) = -\frac{1}{t} + \frac{1}{t^2} \frac{K_{\mathcal{N}}(t)}{K_{\mathcal{N}}^\prime(t)}
+ \frac{1}{t} \frac{K_{\mathcal{N}}(t) K_{\mathcal{N}}^{\prime \prime}(t)}{K_{\mathcal{N}}^\prime(t)^2}.   $$
As $t \rightarrow 1$ the second term vanishes.
To evaluate that limit for the third term, we apply l'H\^opital's rule twice to the expression for $m$,
which yields
$$ m = \lim_{t\rightarrow 1^{-}} \frac{\log K_{\mathcal{N}}(t)}{\log (1-t)} =
\lim_{t\rightarrow 1^{-}} \left(1 - \frac{K_{\mathcal{N}}(t) K_{\mathcal{N}}^{\prime \prime}(t)}{K_{\mathcal{N}}^\prime(t)^2} \right)^{-1}, $$
and hence,
$$  \lim_{t\rightarrow 1^{-}} \frac{K_{\mathcal{N}}(t) K_{\mathcal{N}}^{\prime \prime}(t)}{K_{\mathcal{N}}^\prime(t)^2} = 1 - \frac{1}{m}.  $$
Combining these limits, we obtain
$$ \varphi_{\mathcal{N}}^\prime(1) =\lim_{t\rightarrow 1^{-}}\varphi_{\mathcal{N}}^\prime(t) = -\frac{1}{m}.    $$
Finally, let $y_{\mathcal{N}}(t) = K_{\mathcal{N}}(t)/t$, which allows us to re-write $\varphi_{\mathcal{N}}(t)$ as
$$ \varphi_{\mathcal{N}}(t) = 1 - \frac{y_{\mathcal{N}}(t)}{t\,y_{\mathcal{N}}^{\prime}(t)+y_{\mathcal{N}}(t)}. $$
Equivalently, this gives
$$\frac{y_{\mathcal{N}}^{\prime}(t)}{y_{\mathcal{N}}(t)} = \frac{1}{t} \frac{\varphi_{\mathcal{N}}(t)}{(1-\varphi_{\mathcal{N}}(t))}. $$
Integrating from $0$ to $t$ and noting that $y_{\mathcal{N}}(0)=\lim_{t \to 0} K_{\mathcal{N}}(t)/t = K_{\mathcal{N}}'(0) =1$,
the integration constant vanishes, yielding,
$$ \log y_{\mathcal{N}}(t)=\int_{0}^{t} \frac{1}{s} \frac{\varphi_{\mathcal{N}}(s)}{(1-\varphi_{\mathcal{N}}(s))}\textrm{d}s. $$
\end{proof}

\section{Bibliography}

\begin{description}

\item[] Aczel, J. (1966). \emph{Lectures on Functional Equations and their Applications}. Academic Press.





\item[]
Cowen, C.C. (1984). Commuting analytic functions. \emph{Transactions of the American Mathematical Society}, 283, 685-695.



\item[]
Feller, W. (1943). On a general class of ``contagious" distributions. \emph{Annals of Mathematical Statistics}, 14, 389-400.



\item[]
Gurland, J. (1957). Some interrelations among compound and generalized distributions. \emph{Biometrika}, 44, 265-268.

\item[]
Gurland, J. (1958). A general class of contagious distributions. \emph{Biometrics}, 14, 229-249.



\item[]
Johnson, N.L., Kemp, A.W., Kotz, S. (2005). \emph{Univariate Discrete Distributions, 3rd Ed}. Chapter 9. New York: Wiley.


\item[]
Karlin, S., McGregor, J. (1968a). Embeddability of discrete time simple branching processes into continuous time branching processes.
\emph{Transactions of the American Mathematical Society}, 132, 115-136.

\item[]
Karlin, S., McGregor, J. (1968b). Embedding iterates of analytic functions with two fixed points into continuous groups.
\emph{Transactions of the American Mathematical Society}, 132, 137-145.


\item[]
Kuczma, M. (1968). \emph{Functional equations in a single variable}. Monografie Matematyczne. Warszawa: PWN-Polish Scientific Publishers.

\item[]
Kuczma, M., Choczewski, B., and Ger, R. (1990). \emph{Iterative Functional Equations}. Cambridge University Press.



\item[]
Marshall, A.W., Olkin, I. (1997). A new method for adding a parameter to a family of distributions with application to the exponential
and Weibull families. \emph{Biometrika}, 84, 641-652.

\item[]
Neyman, J. (1939). On a new class of ``contagious" distributions applicable in entomology and bacteriology.
\emph{Annals of Mathematical Statistics}, 10, 35-57.


\item[]
Pranger, W. (1970). Iterations of functions analytic on a disk. \emph{Aequationes Math.}, 4, 203-204.

\item[]
Shaked, M. (1975). On the distribution of the minimum and of the maximum of a random number of i.i.d. random variables.
In \emph{Statistical Distributions in Scientific Work, Vol I}. ed. G.P. Patil, S. Kotz and J.K. Ord. Reidel, Dordrecht. pp. 363-380.

\item[]
Shaked, M., Wong, T. (1997). Stochastic comparisons of random minima and maxima. \emph{Journal of Applied Probability}, 34, 420-425

\item[]
Valero, J., Ginebra, J. (2025). On statistical model extensions based on randomly stopped extremes.
\emph{SORT, Statistics and Operation Research Transactions,} 49, 43-72. \\
https://doi.org/10.57645/20.8080.02.22



\end{description}

\end{document}